\documentclass[10pt]{article}
\usepackage{amsmath,amssymb,amsfonts,amsthm}
\usepackage{mathrsfs}
\usepackage{indentfirst}
\usepackage[pdftex]{color,graphicx}
\usepackage[pdftex,bookmarks,unicode,colorlinks]{hyperref}
\usepackage{enumerate} % ÂÞÁÐ¸ñÊ½
\usepackage[numbers]{natbib} % ²Î¿¼ÎÄÏ×ÒýÓÃ¸ñÊ½
\usepackage{yfonts} % Fraktur/Gothic fonts
\usepackage{stmaryrd} % ÈýÌõÊú¸Ü·¶Êý
\usepackage{caption} % Í¼Æ¬±êÌâ
\usepackage{tikz-cd}
\usetikzlibrary{matrix,arrows,decorations.pathmorphing}

\makeatletter
\let\@fnsymbol\@arabic
\makeatother

\allowdisplaybreaks[4]

\newtheorem{lemma}{Lemma}
\newtheorem{remark}{Remark}
\newtheorem{theorem}{Theorem}
\newtheorem{proposition}{Proposition}

\newtheorem*{example}{Example}
\newtheorem{corollary}{Corollary}
\def\e{\varepsilon}

\def\cC{{\mathcal C}}

\newcommand{\R}{\mathbb R}
\renewcommand{\L}{\mathcal L}
\newcommand{\supp}{\mathrm{supp}\,}

\numberwithin{equation}{section}

\begin{document}

\title{\bf Weak averaging principle for multiscale stochastic dynamical systems driven by stable processes}

\author{\normalsize{\bf Yanjie Zhang$^{a,}\footnote{zhangyj2022@zzu.edu.cn}$, %\footnote{zhangyj18@scut.edu.cn}
Qiao Huang$^{b,c,}\footnote{qiao.huang@ntu.edu.sg}$, %\footnote{qhuang@fc.ul.pt}
Xiao Wang$^{d,}\footnote{ xwang@vip.henu.edu.cn}$,
Zibo Wang$^{e,} \footnote{ zibowang@hust.edu.cn }$,
and Jinqiao Duan$^{f,}\footnote{duan@gbu.edu.cn}$} \\[10pt]
%\footnotesize{${}^a$ School of Mathematics, South China University of Technology, Guangzhou 510641,  China.} 
\footnotesize{${}^a$ Henan Academy of Big Data, Zhengzhou University, Zhengzhou 450001, China.} \\[5pt]
\footnotesize{${}^b$ Division of Mathematical Sciences, School of Physical and Mathematical Sciences, Nanyang Technological University,} \\
\footnotesize{21 Nanyang Link, Singapore 637371.} \\[5pt]
\footnotesize{${}^c$ Group of Mathematical Physics (GFMUL), Department of Mathematics, Faculty of Sciences, University of Lisbon,} \\
\footnotesize{Campo Grande, Edif\'{\i}cio C6, PT-1749-016 Lisboa, Portugal.} \\[5pt]
\footnotesize{${}^d$ School of Mathematics and Statistics, Henan University, Kaifeng 475001, China.} \\[5pt]
\footnotesize{${}^e$ Center for Mathematical Sciences, Huazhong University of Science and Technology, Wuhan 430074, China.} \\[5pt]
\footnotesize{${}^f$ Department of Mathematics \& Department of Physics, Great Bay University, Dongguan 523000, China.}
}

\date{}
\maketitle
\vspace{-0.3in}

\begin{abstract}
We study the averaging principle for a family of multiscale stochastic dynamical systems. The fast and slow components of the systems are driven by two independent stable L\'evy noises, whose stable indexes may be different. The homogenizing index $r_0$ of slow components has a relation with the stable index $\alpha_1$ of the noise of fast components given by $0<r_0<2-2/{\alpha_1}$. By first studying a nonlocal Poisson equation and then constructing suitable correctors, we obtain that the slow components weakly converge to a L\'evy process as the scale parameter goes to zero.
\bigskip\\
  \textbf{AMS 2010 Mathematics Subject Classification:}  
\end{abstract}

%\linenumbers

\section{Introduction.}

Multiscale stochastic dynamical systems arise widely in various areas \cite{GA,GA1,JA,MK,Huang,Huang23}. Finding a coarse-grained model that can effectively characterize the asymptotic behavior of such systems has always been an active research field. Multiscale systems driven by Gaussian noises were originally studied by Khasminski \cite{KR}. Later on, Stroock \cite{SD}, Pardoux \cite{EP} and their collaborators verified the relatively weak compactness of slow components. The former used the martingale approach while the latter used the method of correctors which were constructed via Poisson equations. Then, E et al. \cite{WE} showed the weak and strong convergence results by using the heterogeneous multiscale method. Recently, Cerrai and Freidlin \cite{SM} extended averaging principles to the infinite dimensional case, and Hairer and Li \cite{MHL} investigated a multiscale system driven by fractional Brownian motion.  R\"ockner et al. \cite{Rock} figured out sharp rates, normal derivations, and functional central limits for different multiscale stochastic systems with Brownian noises.

However, random fluctuations in nonlinear systems are sometimes non-Gaussian. A number of authors have studied multiscale dynamical systems driven by stable L\'evy noises. To name a few, Yin and his collaborators \cite{bg} discovered that a stochastic partial differential equation (SPDE) or an SPDE with switching arose as the limit of a multiscale system. The authors in \cite{YZ} combined averaging principles with stochastic filtering problems and showed the convergence of filters. The paper \cite{XS1} studied both strong and weak convergence with different rates. Notably, these results were all based on the fact that the ratio of time between fast and slow components is of order $O(1/\varepsilon)$. A natural and important question is: for multiscale stochastic systems with homogenization terms involved in slow components, how to obtain the corresponding averaged systems when the fast-slow ratio of time is of order $O(1/{\varepsilon}^2)$?

In this paper, we consider the following coupled multiscale system in $\mathbb{R}^{n+m}$:
\begin{equation}
\label{slo0}
\left\{
\begin{aligned}
  d{X^{\varepsilon}_t}& =\frac{1}{\varepsilon^2}b({X^{\varepsilon}_t}, Y^{\varepsilon}_t)dt + \frac{1}{\varepsilon^{\frac{2}{\alpha_1}}}dL_t^{{\alpha_1}}, \;\;  X^{\varepsilon}_0=x\in \mathbb{R}^n, \\
  d{Y^{\varepsilon}_t}&= F({X^{\varepsilon}_t},{Y^{\varepsilon}_t})dt + \frac{1}{\varepsilon^{r_0}}G({X^{\varepsilon}_t},{Y^{\varepsilon}_t})dt + dL_t^{{\alpha}_2}, \;\; Y^{\varepsilon}_0=y \in \mathbb{R}^m,
\end{aligned}
\right.
\end{equation}
where $b, F, G$ are Borel measurable functions on $\mathbb{R}^n \times \mathbb{R}^m$,  the two noises  $L^{{\alpha _i}}, i=1,2$  are independent symmetric $\alpha_i$-stable ($1<\alpha_i <2$) L\'evy processes with triplets $(0,0,{\nu_i})$, and $\nu_i$'s are symmetric $\alpha_i$-stable L\'evy measures (e.g. \cite[Section 14]{Sat99}) on $\R^n$ and $\R^m$ respectively, $\varepsilon>0$ is the scaling parameter, the homogenizing index $r_0$ in slow component satisfies
\begin{equation}\label{index}
  0 <r_0 < 2-2/{\alpha_1}.
\end{equation}
The solution process $({X^{\varepsilon}},{Y^{\varepsilon}})$ is an ${\mathbb{R}^n} \times {\mathbb{R}^m}$-valued process, in which $X^\e$ is called the fast component and $Y^\e$ the slow one. The detailed assumptions on coefficients can be found in Section \ref{sec-2}. The infinitesimal generator $\mathcal{L}^{\varepsilon}$ of the solution $({X^{\varepsilon}},{Y^{\varepsilon}})$ has the following form
\begin{equation*}
\mathcal{L}^{\varepsilon}:=\frac{1}{\varepsilon^2} \left[ -\left(-\Delta_{x}\right)^{\frac{\alpha_1}{2}}+b \cdot \nabla_{x} \right] +\frac{1}{\varepsilon} G \cdot \nabla_{y} + \left[ -\left(-\Delta_{y}\right)^{\frac{\alpha_2}{2}}+F \cdot \nabla_{y} \right].
\end{equation*}
%where ${\mathbb{L}}_{0}$ is given by
%\begin{equation}
%{\mathbb{L}}_{0}:=,
%\end{equation}
%and
%\begin{equation}
%\begin{aligned}
%~~~~{\mathbb{L}}_{1}&:=, \\
%~~~~{\mathbb{L}}_{2}&:=.
%\end{aligned}
%\end{equation}

The aim of this paper is to study the weak convergence of the averaging principle for system \eqref{slo0}. Our work is divided into two parts. In the first part, we will examine the well-posedness of a nonlocal Poisson equation associated with an ergodic jump process, and establish the probability representation of its solution. This nonlocal Poisson equation is usually called the \emph{cell problem} for the averaging problem. In the second part, we will prove the weak averaging principle by constructing suitable ``correctors'' via the nonlocal Poisson equation studied in the first part. We will show that the slow component $Y^\e$ of \eqref{slo0} weakly converges to an averaged process as the scale parameter $\varepsilon$ tends to zero. % The tightness for the slow component helps us to obtain the limit process.
The main result is as follows.

\begin{theorem}
\label{theorem2}
Under Hypotheses ($\bf{A_{b}}$), ($\bf{A_{F}}$), ($\bf{A_{G1}}$) and ($\bf{A_{G2}}$) given in Section \ref{sec-2},  for any $r_0$ satisfying \eqref{index},  the slow component $Y^{\varepsilon}$ converges weakly to a limit process $Y$ as $\varepsilon$  goes to zero. Moreover, the limit process $Y$ is the unique solution to the martingale problem associated with the following operator
\begin{equation*}
\label{l2}
\mathcal{L}_2 = -\left(-\Delta_y\right)^{\frac{\alpha_2}{2}} + {\bar{F}}(y)\cdot \nabla_y,
\end{equation*}
where $\bar F$ is the homogenized drift given by $\bar{F}(y)=\int_{\mathbb{R}^n} F(x,y){\mu^{y}}(dx)$,
and $\mu^{y}$ is an invariant measure (whose existence will be ensured in Lemma \ref{ergodicity0}) of the corresponding frozen equation
\begin{equation*}
dX_t=b(X_t, y)dt+dL^{\alpha_1}_{t}, ~~~X_0=x \in \mathbb{R}^n,
\end{equation*}
that is, for every $0\le t_0<t\le T$,  $\phi\in C^{\infty}_{0}(\mathbb{R}^m)$ and any bounded function  $\Phi$ on $\mathbb{D}\left([0,T]; \mathbb{R}^{m}\right)$ which is measurable with respect to the $\sigma$-field $\sigma\left(\omega_s: \omega \in \mathbb{D}\left([0,T]; \mathbb{R}^{m}\right), t_0 \le s \leq t\right)$,
\begin{equation*}
\mathbb{E}\left[\left(\phi(Y_t)-\phi(Y_{t_{0}})-\int^{t}_{t_0} \mathcal{L}_2\phi(Y_s)ds\right)\Phi(Y)\right]=0.
\end{equation*}
\end{theorem}

In contrast to the case of Brownian noise in \cite[Theorem 3]{EP}, Theorem \ref{theorem2} shows that the homogenization term $G$ in the slow component $Y^{\varepsilon}$ does not affect the limit process $Y$. This phenomenon is due to the feature of stable L\'evy noise, as we will see in Subsection \ref{4.2} where a corrector of order $\varepsilon^{2-r_0}$ will be constructed to prove the tightness of $\{Y^{\varepsilon}\}_{\varepsilon >0}$.  Our argument is slightly different from existing works. Firstly, a Poisson equation in $\mathbb{R}^{n}$ for an elliptic operator corresponds to an ergodic diffusion process in the case of Brownian noise. This hints that a nonlocal Poisson equation for a nonlocal operator should correspond to an ergodic diffusion process driven by $\alpha$-stable noise. Secondly, the technique of Poisson equations will be used not only to obtain the optimal rate for the strong convergence as in the Brownian driven case, but also to prove the tightness of slow components for the weak convergence. We also remark that for technical reasons (cf. the proof of Lemma \ref{lipc}), we require in the definition of the martingale solution of $\mathcal L_2$ that $\phi$ is in $\mathcal C^{\infty}_{0}$, the space of infinitely differentiable functions with compact support. This is quite standard in the formulation of martingale problems of L\'evy-type operators, cf \cite{BSW13}.
%Thirdly, the limit process $Y$ is the unique solution of the ``formal'' martingale problem associated with the nonlocal operator $\mathcal L_2$. The difference is that we require the function $\phi \in C^{\infty}_{0}$ rather than $\phi \in C^{2}_{0}$ in the present paper. Finally, the operator $\mathcal{L}_2$ in Theorem \ref{theorem2} can be viewed as the infinitesimal generator of the frozen equation in Theorem 1.
Compared with the results of Brownian driven case \cite{EP, Rock}, our averaging problem has some difficulties to overcome. Firstly, since the L\'evy noise in our case is not square integrable, the solution $\left(X^{\varepsilon}_t, Y^{\varepsilon}_t\right)$ has finite $p$-th moment only for $p \in (0,\alpha_1 \wedge \alpha_2)$ (cf. \cite[Theorem 25.3]{Sat99}). %Hence the classical Khasminskii's time discretization technique may not be applicable.
Secondly, the cell problem in our case, i.e., the Poisson equation associated with the generator of a jump process, is \emph{nonlocal}. Thirdly, the homogenizing index $r_0$ is now related to the stable index $\alpha_1$ and the relation \eqref{index} is exactly what we need to seek.

\begin{example}

Consider the following fast-slow model
\begin{equation*}
\left\{
\begin{aligned}
d{X^{\varepsilon}_t}& =-\frac{{X^{\varepsilon}_t}}{\varepsilon^2}dt + \frac{1}{\varepsilon^{\frac{2}{\alpha_1}}}dL_t^{{\alpha}_1}, \;\;  X^{\varepsilon}_0=x\in \mathbb{R}, \\
d{Y^{\varepsilon}_t}&= \frac{\sin X^{\varepsilon}_t}{\varepsilon^{r_0}}dt + dL_t^{{\alpha}_2}, \;\; Y^{\varepsilon}_0=y \in \mathbb{R},
\end{aligned}
\right.
\end{equation*}
where $b(x)=-x$, $\alpha_1=1.5$, $F(x,y)=0$,  $G(x,y)=\sin x$ and $r_0=2-\frac{2}{1.5}=\frac{2}{3}$.
It is easy to justify that $b, F, G$ satisfy Hypotheses ($\bf{A_{b}}$), ($\bf{A_{F}}$), ($\bf{A_{G1}}$),($\bf{A_{G2}}$). Using a result in \cite{ar}, we find the invariant measure   $\mu(dx)=\rho(x)dx$ with density
\begin{equation*}
\rho(x)=\frac{1}{{2\pi}}\int_{\mathbb{R}}e^{ix\xi}e^{-\frac{1}{\alpha}|\xi |^{\alpha}}d\xi=\frac{1}{\pi}\int ^{\infty}_{0}\cos x\xi \cdot e^{-\frac{1}{\alpha}\xi ^{\alpha}}d\xi.
\end{equation*}
%Define
%\begin{equation}
%\left\{
%\begin{aligned}
%\widetilde{G}(x,y)&=\int^{\infty}_{0}\mathbb{E}[G(X^{x}_t,y)]dt,\\
%dX^{x}_t&=-X^{x}_tdt +dL^{\alpha_1}_{t}, ~~X^{x}_0=x \in \mathbb{R}.
%\end{aligned}
%\right.
%\end{equation}
It follows from Theorem \ref{theorem2} that the averaged equation for $Y^{\varepsilon}_t$ is
\begin{equation*}
d\bar{Y}_t= dL^{\alpha_2}_{t}, ~~\bar{Y}_0=y,
\end{equation*}
and the slow component $Y^{\varepsilon}$ converges weakly to a process $\bar{Y}$ as the scale parameter $\varepsilon$  goes to zero.

\end{example}

The paper is organized as follows.  In Section \ref{sec-2},  we will list all assumptions for the coefficients that are required by our main result.
Section \ref{sec-3} is devoted to the existence of a nonlocal Poisson equation in the whole space. In Section \ref{sec-4}, we will use the nonlocal Poisson equation to construct a corrector and apply it to study the weak averaging principle for the multiscale system \eqref{slo0}. Section \ref{sec-5} is reserved for some concluding remarks. Some tedious proofs of lemmas are left in Appendix \ref{app-1}.

To end the introduction, we list the notations that will be used frequently in the sequel. The letter $C$ denotes a positive constant whose value may change from one line to another. The notation $C_p$ is used to emphasize that the constant depends only on a parameter $p$, while $C(\cdots)$ is used for the case that there is more than one parameter. We use $\otimes$, $\langle \cdot, \cdot \rangle$, and $|\cdot|$ to denote the tensor product, inner product, and norm in Euclidean spaces, respectively. We also use $\nabla$ to denote the gradient operator in Euclidean spaces, and $\nabla_x$ to emphasize that the gradient is with respect to the variable $x$. For any positive integer $k,l$ and probability measure $\mu$, we introduce the following function spaces:
\begin{equation*}
\begin{aligned}
&\mathcal{B}_{b}(\mathbb{R}^{n}):=\left\{f: \mathbb{R}^{n}\rightarrow \mathbb{R}\mid f\text{ is bounded Borel measurable}\right\},\\
&\mathcal{C}_{0}(\mathbb{R}^{n}):=\left\{f: \mathbb{R}^{n}\rightarrow \mathbb{R}\mid f\text{ is continuous and has compact support}\right\},\\
&\mathcal{C}^{\mu}_{0}(\mathbb{R}^{n}):=\left\{f \in\cC_{0}(\R^n)\mid f\text{ is centered with respect to the measure $\mu$, i.e., $\textstyle{\int_{\mathbb{R}^{n}} f(x) \mu(dx)} = 0$}\right\},\\
&\cC^{k}(\mathbb{R}^{n}):=\left\{f: \mathbb{R}^{n}\rightarrow \mathbb{R}\mid f\text{ and all its partial  derivatives up to order $k$ are continuous} \right\},\\
&\cC^{k}_{b}(\mathbb{R}^{n}):=\left\{f \in \cC^{k}(\mathbb{R}^{n})\mid\text{ for  $ 1\leq i \leq k$, the $i$-th order partial  derivatives of $f$ are bounded} \right\},\\
&\cC^{k,l}_{b}(\mathbb{R}^{n}\times \mathbb{R}^{m} ):=\left\{f(x,y)\mid \text{ for $1\leq |\beta_1| \leq k$ and $1\leq |\beta_2|\leq l$, $\nabla^{\beta_1}_{x}\nabla^{\beta_2}_{y}f $ is uniformly bounded} \right\}.
\end{aligned}
\end{equation*}
We equip the space $\mathcal{C}_{0}(\mathbb{R}^{n})$ with the norm $\|f\|_{0}=\sup_{x \in \mathbb{R}^n}|f(x)|$ and $\cC^{k}(\mathbb{R}^{n})$ with the norm $\|f\|_{k}=\|f\|_{0}+\sum_{j=1}^k\|\nabla^{\otimes j}f\|$. It is clear that $(\cC^{k}(\mathbb{R}^{n}),\|\cdot\|_{k})$ is a Banach space. \\

\section {Assumptions.}\label{sec-2}

In this section, we collect all assumptions we need for the coefficients of system \eqref{slo0}. First of all, we need some regularity assumptions for drifts $b$, $F$ and $G$.

%($\bf{A_{b^{'}}}$): For any $ x_1, x_2 \in \mathbb{R}^n$ and $r > 0$ such that $|x_1-x_2|=r$. Set
%\begin{equation*}
%\hat{g}(r) = \inf \left\{-\frac{\left\langle b(x_1,y)-b(x_2,y),x_1-x_2\right\rangle}{|x_1-x_2|^2}:  b\in C^{2,2}_{b} ~~\text{and}~~ |b(0,y)|< \infty \right\}.
%\end{equation*}
%The function $\hat{g}$ is a continuous function satisfying
%
%\begin{equation}
%\lim _{r\rightarrow \infty} \inf \hat{g}(r)>0.
%\end{equation}

\noindent ($\bf{A_{b}}$):
%For all function $ b\in C^{2,2}_{b}$ and $|b(0,y)|< \infty$, the function $\tilde{{g}}$ defined by
%\begin{equation*}
%\tilde{{g}}(r):= \inf \left\{-\frac{\left\langle b(x_1,y)-b(x_2,y),x_1-x_2\right\rangle}{|x_1-x_2|^2}: x_1, x_2 \in \mathbb{R}^{n}, |x_1-x_2|=r, y\in \mathbb{R}^{m} \right\}
%\end{equation*}
%satisfies
%\begin{equation}
%\liminf_{r\rightarrow \infty} \tilde{g}(r)>0.
%\end{equation}
Suppose that $ b\in \mathcal{C}^{2,2}_{b}(\mathbb{R}^{n}\times \mathbb{R}^{m} )$, and there exists a positive constant $\gamma$ such that for all $x_1, x_2 \in \mathbb{R}^{n}$,
\begin{equation}
\label{0disb}
  \sup_{y\in \mathbb{R}^{m}} |b(0,y)|< \infty, \quad \sup_{y\in \mathbb{R}^{m}} \langle b(x_1,y)-b(x_2,y), x_1-x_2 \rangle \le -\gamma|x_1-x_2|^2.
\end{equation}

%Notably, the assumption ($\bf{A_{b}}$) will be applied to \eqref{slo0}.

%($\bf{A_{f}}$): The function $f$ is  ``centered'', i.e.,
%\begin{equation}
%\int_{\mathbb{R}^n} f(x)\mu(dx)=0,
%\end{equation}
%where $\mu$ is the invariant probability measure of $X$.
%

\noindent ($\bf{A_{F}}$): The function $F$ is both Lipschitz and bounded, i.e., there exists a positive constant $K_1$ such that for all $x,x_1, x_2 \in \mathbb{R}^{n}$ and $y,y_1, y_2 \in \mathbb{R}^{m}$,
\begin{equation*}
| F(x_1,y_1)-F(x_2,y_2)|^{2} \leq K_1 ( | x_1-x_2|^2+| y_1-y_2|^2 ),
\end{equation*}
and
\begin{equation*}
\left|F(x,y)\right|\leq K_1.
\end{equation*}

\noindent ($\bf{A_{G1}}$):
The function $G $ satisfies the following conditions: there exists a positive constant $K_2$ such that for all $x,x_1, x_2 \in \mathbb{R}^{n}$ and $y,y_1, y_2 \in \mathbb{R}^{m}$,
\begin{equation*}
| G(x_1,y_1)-G(x_2,y_2)|^{2} \leq K_2 ( | x_1-x_2|^2+| y_1-y_2|^2 ),
\end{equation*}
and
\begin{equation*}
\begin{aligned}
&\left|G(x,y)\right|\leq K_2,
%~~\sup_{x,y}|\nabla_{x}G(x,y)|\leq K_2, \\
%&\sup_{y}|\nabla_{y}G(x,y)|\leq K_2(1+|x|),
~~\sup_{x,y}|\nabla^2_x G(x,y)|\leq K_2,\\
&\sup_{x, y}|\nabla^2_{x}\nabla_y G(x,y)|\leq K_2, ~~
\sup_{x, y}|\nabla^3_x G(x,y)|\leq K_2.
\end{aligned}
\end{equation*}

We also need a centering assumption for the function $G$, as follows.

\noindent ($\bf{A_{G2}}$): For each $y \in \mathbb{R}^{m}$, the function $G(\cdot,y)$ is centered with respect to $\mu^y$, i.e.,
\begin{equation*}
\int_{\mathbb{R}^{m}} G(x,y)\mu^y(dx) =0,
\end{equation*}
where $\mu^y$ is an invariant measure of an ergodic Markov process $X^{x,y}$ (see \eqref{itro} below).

\section{The nonlocal Poisson equation.}\label{sec-3}

In this section, we study the following nonlocal Poisson equation in $\mathbb{R}^{n}$,
\begin{equation}
\label{poisson}
\mathcal{L}u(x)=-f(x),
\end{equation}
where $f$ is a given function on $\R^n$ and
\begin{equation*}
\label{gene}
\begin{aligned}
\mathcal{L}=-\left(-\Delta_x\right)^{\frac{\alpha_1}{2}}+\sum_{i} b_i(x)\partial_{x_i}.
\end{aligned}
\end{equation*}
The operator $\mathcal{L}$ is the formal generator of the following stochastic differential equation (SDE) on the probability space $(\Omega,\mathcal F, \mathbb P)$,
\begin{equation}
\label{0nsde}
dX^x_t=b(X^x_t)dt +dL^{\alpha_1}_t, \quad X^x_0=x \in \mathbb{R}^n.
\end{equation}
We denote by $\{P_t\}_{t\ge0}$ the semigroup generated by $\mathcal{L}$ (or by $X^x$) on $\cC_0(\R^n)$.
%and $\mu$ is the invariant probability measure of the solution process $X$ of

Here we need to give a regularity assumption for the drift $b$ in \eqref{0nsde}, which is only valid for this section.

\noindent ($\bf{A'_{b}}$): Suppose that $ b\in \mathcal{C}^2_{b}(\mathbb{R}^{n})$, and the function $\hat{g}$ defined by
\begin{equation}
\label{0conditionb0}
\hat{g}(r):= \inf \left\{-\frac{\left\langle b(x_1)-b(x_2),x_1-x_2\right\rangle}{|x_1-x_2|^2}: x_1, x_2 \in \mathbb{R}^{n}, |x_1-x_2|=r  \right\}
\end{equation}
satisfies
\begin{equation*}
\liminf_{r\rightarrow \infty} \hat{g}(r)>0.
\end{equation*}

%\vspace{2mm}
%Note that the assumption ($\bf{A^{'}_{b}}$) will be applied to the following stochastic dynamical systems,
%\begin{equation}
%\label{nsde}
%dX_t=b(X_t)dt +dL^{\alpha_1}_t, ~~ X_0=x \in \mathbb{R}^n.
%\end{equation}

\begin{remark}%\label{remark1}
Condition \eqref{0conditionb0} is a weak version of the dissipative condition. To be precise, it is equivalent to that for any $M>0$ there exists a positive constant $R>0$ such that for
all $x_1, x_2 \in \mathbb{R}^n$  with $|x_1-x_2|\geq R$,
\begin{equation}\label{0diss}
\left \langle b(x_1)-b(x_2), x_1-x_2\right \rangle \leq -M |x_1-x_2|^2.
\end{equation}
As a consequence, by letting $x_2 = 0$ and using Young's inequality, we have
\begin{equation}\label{diss}
\begin{aligned}
\left \langle b(x), x\right \rangle &\leq |b(0)| |x| -M |x|^2 \leq -\frac{M}{2}|x|^2+\frac{1}{M}|b(0)|^2.
\end{aligned}
\end{equation}
\end{remark}

Under Hypothesis ($\bf{A'_{b}}$), SDE \eqref{0nsde} has a unique global strong solution $X^x$ \cite[Theorem 6.2.3]{Da}.
One expects that the solution of \eqref{poisson} has the following probabilistic representation, whose proof will be given in Subsection \ref{sec-3-2},
\begin{equation}\label{rep_u}
u(x)=\int^{\infty}_0 \mathbb{E} f(X^x_s)ds.
\end{equation}

\subsection{\bf{Invariance and ergodicity.}}

\begin{lemma}
Let ($\bf{A'_{b}}$) hold. Then for any $1 \leq p < \alpha_1$ and $T>0$, we have
\begin{align}
\mathbb{E}\left(\sup_{0 \leq t \leq T}|X^x_t|^p\right) & \leq C_{p} \left(T^{\frac{p}{\alpha_1}}\vee T^{1-\frac{2}{\alpha_1}+\frac{p}{\alpha_1}}\right)+|x|^p, \label{moment0} \\
\sup_{t \ge 0}\mathbb{E}|X^x_t|^p & \le C_p (1+|x|^p). \label{moment1}
\end{align}
\end{lemma}
\begin{proof}
We follow the lines of \cite[Lemma A.2]{XS1}. We first prove \eqref{moment0}. To this purpose, we define an auxiliary function for each $T>0$,
\begin{equation*}
U_T(z):=\left(|z|^2+T^\frac{2}{\alpha_1} \right)^{\frac{p}{2}}.
\end{equation*}
Then for any $z\in \mathbb{R}^n$, we get
\begin{equation}
\label{onederiva}
\left | \nabla U_T(z) \right| =\frac{p|z|}{(T^\frac{2}{\alpha_1}+|z|^2)^{1-p/2}} \leq C_p|z|^{p-1}.
\end{equation}

On the one hand, by the definition of $U_{T}(z)$, we know
\begin{equation*}
\begin{aligned}
\frac{\partial U_{T}(z)}{\partial z_i}&=pz_i\left(T^{\frac{2}{\alpha_1}}+z^2_1+z^2_2+\cdot\cdot\cdot+z^2_n\right)^{\frac{p}{2}-1}.
\end{aligned}
\end{equation*}
Thus we have
\begin{equation*}
\nabla U_{T}(z)=\frac{pz}{\left(T^{\frac{2}{\alpha_1}}+|z|^2\right)^{1-\frac{p}{2}}}.
\end{equation*}
Similarly, by a simple calculation, we have
\begin{eqnarray}
\label{z1}
\nabla\left[\frac{pz_1}{\left(T^{\frac{2}{\alpha_1}}+|z|^2\right)^{1-\frac{2}{p}}}\right] =
\begin{pmatrix}
  \frac{p}{\left(T^{\frac{2}{\alpha_1}}+|z|^2\right)^{1-\frac{p}{2}}}-\frac{p(2-p)z^2_1}{\left(T^{\frac{2}{\alpha_1}}+|z|^2\right)^{2-p/2}}     \\
   -\frac{p(2-p)z_1z_2}{\left(T^{\frac{2}{\alpha_1}}+|z|^2\right)^{2-p/2}}    \\
         \vdots  \\
	  -\frac{p(2-p)z_1z_{n-1}}{\left(T^{\frac{2}{\alpha_1}}+|z|^2\right)^{2-p/2}} \\
	-\frac{p(2-p)z_1z_{n}}{\left(T^{\frac{2}{\alpha_1}}+|z|^2\right)^{2-p/2}}
\end{pmatrix}
\end{eqnarray}
and
\begin{eqnarray}
\label{z2}
\nabla\left[\frac{pz_2}{\left(T^{\frac{2}{\alpha_1}}+|z|^2\right)^{1-\frac{2}{p}}}\right] =
\begin{pmatrix}
-\frac{p(2-p)z_1z_2}{\left(T^{\frac{2}{\alpha_1}}+|z|^2\right)^{2-p/2}}     \\
   \frac{p}{\left(T^{\frac{2}{\alpha_1}}+|z|^2\right)^{1-\frac{p}{2}}}-\frac{p(2-p)z^2_2}{\left(T^{\frac{2}{\alpha_1}}+|z|^2\right)^{2-p/2}}    \\
         \vdots  \\
	  -\frac{p(2-p)z_2z_{n-1}}{\left(T^{\frac{2}{\alpha_1}}+|z|^2\right)^{2-p/2}} \\	
-\frac{p(2-p)z_2z_n}{\left(T^{\frac{2}{\alpha_1}}+|z|^2\right)^{2-p/2}}
\end{pmatrix}
\end{eqnarray}
\begin{eqnarray*}
~~~~~~~~~~~~~~~~\vdots
\end{eqnarray*}
\begin{eqnarray}
\label{zn}
\nabla\left[\frac{pz_n}{\left(T^{\frac{2}{\alpha_1}}+|z|^2\right)^{1-\frac{2}{p}}}\right] =
\begin{pmatrix}
 -\frac{p(2-p)z_1z_n}{\left(T^{\frac{2}{\alpha_1}}+|z|^2\right)^{2-p/2}}     \\
   -\frac{p(2-p)z_2z_n}{\left(T^{\frac{2}{\alpha_1}}+|z|^2\right)^{2-p/2}}    \\
         \vdots  \\
	  -\frac{p(2-p)z_{n-1}z_{n}}{\left(T^{\frac{2}{\alpha_1}}+|z|^2\right)^{2-p/2}} \\	
\frac{p}{\left(T^{\frac{2}{\alpha_1}}+|z|^2\right)^{1-\frac{p}{2}}}-\frac{p(2-p)z^2_n}{\left(T^{\frac{2}{\alpha_1}}+|z|^2\right)^{2-p/2}}
\end{pmatrix}
\end{eqnarray}
Combining \eqref{z1}-\eqref{zn} and by the inequality $z_iz_j\leq |z|^2$, we obtain
\begin{equation}
\label{twderiva}
\begin{aligned}
\left|\nabla^2U_T(z)\right|&=\left|\frac{pI_{n\times n}}{(T^\frac{2}{\alpha_1}+|z|^2)^{1-p/2}}-\frac{p(2-p)z\otimes z}{(T^\frac{2}{\alpha_1}+|z|^2)^{2-p/2}}\right|\\
 &\leq \frac{C_{p,n}}{\left(T^\frac{2}{\alpha_1}+|z|^2\right)^{1-p/2}}\\
 &\leq C_{p,n} T^{\frac{2}{\alpha_1}(\frac{p}{2}-1)}.
\end{aligned}
\end{equation}
where $I_{n\times n}$ is the $n\times n$ identity matrix. \\
\\
\smallskip~~~~~ It is clear that $X^{x}_t$ can be rewritten as
\begin{equation*}
X^{x}_t=x+\int^{t}_{0}b(X^{x}_s)ds +\int^{t}_{0}\int_{|z|\leq T^{\frac{1}{\alpha_1}}}z \widetilde{N}^1(dz,ds) +\int^{t}_{0}\int_{|z|> T^{\frac{1}{\alpha_1}}}z {N}^1(dz,ds).
\end{equation*}
Using It\^o's formula, we gain
\begin{equation}\label{Uest}
\begin{aligned}
U_T(X^{x}_t)=&\ U_T(X^x_0)+\int^t_{0}\left\langle b(X^{x}_s), \nabla U_T(X^{x}_s)\right\rangle ds
+\int^t_{0}\int_{|z|\leq T^{\frac{1}{\alpha_1}}}\left[U_T(X^{x}_s+z)-U_T(X^{x}_s)\right]\widetilde{N}^{1}(ds,dz)\\
&\ +\int^t_{0}\int_{|z|\leq T^{\frac{1}{\alpha_1}}}\left[U_T(X^{x}_s+z)-U_T(X^{x}_s)-\langle \nabla U_T(X^{x}_s), z\rangle \right]\nu_{1}(dz)ds \\
&\ +\int^t_{0}\int_{|z| >T^{\frac{1}{\alpha_1}}}\left[U_T(X^{x}_s+z)-U_T(X^{x}_s)\right]N^{1}(ds,dz)\\
:=&\ U_T(X^x_0)+U^{(1)}+U^{(2)}+U^{(3)}+U^{(4)}.
\end{aligned}
\end{equation}
For the term $U^{(1)}$, we use \eqref{diss} and obtain
\begin{equation*}
\begin{aligned}
\mathbb{E}\left[\sup_{0 \leq t \leq T}|U^{(1)}|\right] &= \mathbb{E}\left[\sup_{0 \leq t \leq T}\int^t_{0}\left\langle b(X^{x}_s), \frac{pX^{x}_s}{(T^{\frac{2}{\alpha}}+|X^{x}_s|^2)^{1-\frac{p}{2}}}\right\rangle ds\right] \\
& \leq \int^{T}_{0}\frac{C(p, M, b(0))}{(T^{\frac{2}{\alpha_1}}+|X^{x}_s|^2)^{1-\frac{p}{2}}}ds
 \leq C(p, M, b(0)) T^{1-\frac{2}{\alpha_1}+\frac{p}{\alpha_1}}.
\end{aligned}
\end{equation*}
For the term $U^{(2)}$, by  the Burkholder--Davis--Gundy inequality, \eqref{onederiva} and Young's inequality, we get
\begin{equation*}
\begin{aligned}
\mathbb{E}\left[\sup_{0 \leq t \leq T}|U^{(2)}|\right]  &\leq  C \mathbb{E}\left[\int^{T}_{0}\int_{|z|\leq T^{\frac{1}{\alpha_1}}}\left(\int^{1}_{0}|\nabla U_T(X^{x}_s+z\xi)|^2d\xi\right)|z|^2N^{1}(dz,ds)\right]^{\frac{1}{2}}\\
&\leq C_{p} \mathbb{E}\left\{\int^{T}_{0}\int_{|z|\leq T^{\frac{1}{\alpha_1}}}\left[\int^{1}_{0}\left(|X^{x}_s|^{2p-2}
+|z|^{2p-2}\xi^{2p-2}\right)d\xi\right]z^2N^{1}(dz,ds)\right\}^{\frac{1}{2}}\\
& \leq C_{p}\mathbb{E}\left\{\int^{T}_{0}\int_{|z|\leq T^{\frac{1}{\alpha_1}}}\left(|X^{x}_s|^{2p-2}z^2+|z|^{2p}\right)N^{1}(dz,ds)\right\}^{\frac{1}{2}}\\
& \leq C_p \mathbb{E}\left\{\left(\sup_{0 \leq s \leq T}|X^{x}_t|^{p-1}\right)\left[\int^{T}_{0}\int_{|z|\leq T^{\frac{1}{\alpha_1}}}|z|^2N^{1}(dz,ds)\right]^{\frac{1}{2}}\right\}
+C_p \mathbb{E}\left[\int^{T}_{0}\int_{|z|\leq T^{\frac{1}{\alpha_1}}}|z|^{2p}N^{1}(dz,ds)\right]\\
& \leq \frac{1}{4} \mathbb{E}\left(\sup_{0 \leq s \leq T}|X^{x}_s|^p\right)+C_pT^{\frac{p}{\alpha_1}}.
\end{aligned}
\end{equation*}
For the term $U^{(3)}$, by Taylor's expansion and inequality \eqref{twderiva}, we have
\begin{equation*}
\begin{aligned}
\mathbb{E}\left[\sup_{0 \leq t \leq T}|U^{(3)}|\right]  & \leq C_p T^{\frac{2}{\alpha_1}(\frac{p}{2}-1)}\int^{T}_{0}\int_{|z|\leq T^{\frac{1}{\alpha_1}}}z^2\nu_{1}(dz)ds
 \leq C_p T^{\frac{p}{\alpha_1}}.
\end{aligned}
\end{equation*}
For the term $U^{(4)}$, using again \eqref{onederiva} and Young's inequality, we obtain
\begin{equation}
\label{U4}
\begin{aligned}
\mathbb{E}\left[\sup_{0 \leq t \leq T}|U^{(4)}|\right]&\leq  C \mathbb{E}\left[\int^{T}_{0}\int_{|z|> T^{\frac{1}{\alpha_1}}}\left(\int^{1}_{0}|\nabla U_T(X^{x}_s+z\theta)|d\theta\right)|z|\nu_1(dz)ds\right]\\
& \leq C_p\mathbb{E}\left[\int^{T}_{0}\int_{|z|> T^{\frac{1}{\alpha_1}}}\left(|X^{x}_s|^{p-1}|z|+|z|^{p-1}\right)\nu_1(dz)ds\right]\\
& \leq C_p \mathbb{E}\left[\left(\sup_{0 \leq s \leq T}|X^{x}_s|^{p-1}\right)\left(\int^{T}_{0}\int_{|z|> T^{\frac{1}{\alpha_1}}}|z|\nu_1(dz)ds\right)\right]+C_p \int^{T}_{0}\int_{|z|> T^{\frac{1}{\alpha_1}}}|z|^p\nu_1(dz)ds\\
&\leq \frac{1}{4}\mathbb{E}\left(\sup_{0 \leq s \leq T}|X^{x}_s|^{p-1}\right)
+C_p\left[\int^{T}_{0}\int_{|z|> T^{\frac{1}{\alpha_1}}}|z|\nu_1(dz)ds\right]^p
+C_p \int^{T}_{0}\int_{|z|> T^{\frac{1}{\alpha_1}}}|z|^p\nu_1(dz)ds\\
& \leq \frac{1}{4}\mathbb{E}\left(\sup_{0 \leq s \leq T}|X^{x}_s|^{p-1}\right)+C_pT^{\frac{p}{\alpha_1}}.
\end{aligned}
\end{equation}
Combining \eqref{Uest}--\eqref{U4}, we have the following estimate which yields \eqref{moment0},
\begin{equation*}
\begin{aligned}
\mathbb{E}\left[\sup_{0 \leq t \leq T}U_T(X^x_t)\right]\leq \frac{1}{2} \mathbb{E}\left(\sup_{0 \leq t \leq T}|X^x_t|^p\right)+C_p \left(T^{\frac{p}{\alpha_1}}\vee T^{1-\frac{2}{\alpha_1}+\frac{p}{\alpha_1}}\right)+\mathbb E (|X^x_0|^p).
\end{aligned}
\end{equation*}

To prove \eqref{moment1}, we need to use the special case $T=1$ of $U_T$, i.e., $U_1$. Then for any $z\in \mathbb{R}^n$, we obtain from \eqref{onederiva} and \eqref{twderiva} that
\begin{equation}\label{onederiva0}
\left | \nabla U_1(z) \right| \leq C_p|z|^{p-1}, \quad \left|\nabla^2U_1(z)\right|
 \leq C_p.
\end{equation}
Using It\^o's formula to $U_1(X^{x}_t)$ and then taking expectation, we obtain
%\begin{equation*}
%\begin{aligned}
%\mathbb{E}U_1(X^{x}_t)&= U_1(x)+\mathbb{E}\int^t_{0}\left\langle b(X^{x}_s), \nabla U_1(X^{x}_s)\right\rangle ds
%+\mathbb{E}\int^t_{0}\int_{|z|\leq 1}\left[U_1(X^{x}_s+z)- U_1(X^{x}_s)-\langle \nabla U_1(X^{x}_s), z\rangle \right]\nu_{1}(dz)ds \\
%&\quad\ +\mathbb{E}\int^t_{0}\int_{|z| >1}\left[U_1(X^{x}_s+z)-U_1(X^{x}_s)\right]\nu_{1}(dz)ds,
%\end{aligned}
%\end{equation*}
%which implies
\begin{equation}\label{est-6}
\begin{aligned}
\frac{d\mathbb{E}U_1(X^{x}_t)}{dt}&=\mathbb{E}\left\langle b(X^{x}_s), \nabla U_1(X^{x}_s)\right\rangle+\mathbb{E}\int_{|z|\leq 1}\left[U_1(X^{x}_s+z)-U_1(X^{x}_s)-\langle \nabla U_1(X^{x}_s), z\rangle \right]\nu_{1}(dz)\\
&\quad\ +\mathbb{E}\int_{|z| >1}\left[U_1(X^{x}_s+z)-U_1(X^{x}_s)\right]\nu_{1}(dz)\\
&:=M_1+M_2+M_3.
\end{aligned}
\end{equation}
For the term $M_1$, by \eqref{onederiva} and \eqref{0diss}, there exists a positive constant $\kappa>0$ such that
\begin{equation*}
\begin{aligned}
M_1 &= \mathbb{E}\left[\left\langle b(X^{x}_s)-b(0), \frac{pX^{x}_s}{(1+|X^{x}_s|^2)^{1-\frac{p}{2}}}\right\rangle \right]+ \mathbb{E}\left[\left\langle b(0), \frac{pX^{x}_s}{(1+|X^{x}_s|^2)^{1-\frac{p}{2}}}\right\rangle \right]\\
 &\leq \mathbb{E}\left[ \frac{-pM|X^{x}_s|^2+C(p,|b(0)|)|X^{x}_s|}{\left(1+|X^{x}_s|^2\right)^{1-\frac{p}{2}}}\right]\\
 &\leq -\kappa \mathbb{E}U_1(X^{x}_s)+C(p, |b(0)|).
\end{aligned}
\end{equation*}
For the term $M_2$, by \eqref{onederiva0}, we get
\begin{equation*}
\begin{aligned}
|M_2|\leq  C_{p} \mathbb{E}\left[\int_{|z|\leq 1}|z|^2\nu_{1}(dz)\right] \leq  C_{p}.
\end{aligned}
\end{equation*}
For the term $M_3$, using again \eqref{onederiva0} and Young's inequality, we obtain
\begin{equation}
\label{U40}
\begin{aligned}
|M_3|&\leq \mathbb{E}\left[\int_{|z|>1}\left(\int^{1}_{0}|\nabla U_1(X^{x}_s+z\theta)|d\theta\right)|z|\nu_1(dz)\right]\\
& \leq \mathbb{E}\left[\int_{|z|> 1}\left(|X^{x}_s|^{p-1}+|z|^{p-1}\right)\nu_1(dz)\right]\\
& \leq \frac{\kappa}{2} \mathbb{E}U_1(X^{x}_s)+C_p.
\end{aligned}
\end{equation}
Combining \eqref{est-6}--\eqref{U40}, we have
\begin{equation*}
\begin{aligned}
\frac{d\mathbb{E}U_1(X^x_t)}{dt}\leq -\frac{\kappa}{2} \mathbb{E}U_1(X^{x}_t)+C_p.
\end{aligned}
\end{equation*}
By the comparison theorem, we have
\begin{equation*}
\begin{aligned}
\mathbb{E}U_1(X^x_t)& \leq  e^{\frac{-\kappa t}{2}}\left(1+|x|^2\right)^{\frac{p}{2}}
+C_{p, \kappa}\int^{t}_{0}e^{\frac{-(t-s)\kappa}{2}}ds \\
& \leq C_{p}\left(1+|x|^p\right).
\end{aligned}
\end{equation*}
This yields the desired result.
%\begin{equation*}
%\mathbb{E}\left[\sup_{0 \leq t \leq T}|X^x_t|^p\right]\leq C_{p} \left(T^{\frac{p}{\alpha_1}}\vee T^{1-\frac{2}{\alpha_1}+\frac{p}{\alpha_1}}\right)+|x|^{p}.
%\end{equation*}
\end{proof}

In the following lemma, we will study the invariance and ergodicity of SDE \eqref{0nsde}.

\begin{lemma}\label{ergodicity0}
 Let ($\bf{A'_{b}}$) hold. %and assume further that the semigroup $\{P_t\}_{t\ge0}$ preserves finite first moment.
 Then \\
 (i) the process $X$ possesses an invariant distribution $\mu$ on $\R^n$ which satisfies $\int_{\R^n} |z|^p \mu(dz) <\infty$ for any $1 \leq p < \alpha_1$; \\
 (ii) there exist a constant $\rho>0$, such that for each $x\in\R^n$, there exist positive constants $C_{x}$, we have
  \begin{equation}
  \label{est-1}
    \left| P_t f(x)-\int_{\R^n} f(y)\mu(dy) \right| \le C_{x}\|f\|_0e^{-\rho t}, \quad \forall f \in \mathcal{B}_{b}(\mathbb{R}^n),\ \forall t\ge0.
  \end{equation}
\end{lemma}

\begin{proof}
  It follows from Lemma \ref{moment0} that the semigroup $\{P_t\}_{t\ge0}$ preserves a finite first moment, i.e., if a measure $\eta$ has a finite first moment, then the measure $P^*_t \eta$ also has a finite first moment for all $t\ge0$. The existence of invariant distribution $\mu$ follows from \cite[Corollary 1.8]{Maj}. The moment estimate for $\mu$ is implied by the estimate \eqref{moment1}. Also by the same corollary, for all $r>0$, there exists a concave function $\phi:[0,\infty) \to [0,\infty)$ with $\phi(0)=0$ and $\phi(r)>0$, and two positive constants $C$ and $\rho$, such that for all nonnegative $t$ and any probability measure $\eta$, we have
  \begin{equation*}
    \|\mu-P^*_t \eta\|_{\text{TV}} \le C e^{-\rho t} W_\phi (\mu,\eta),
  \end{equation*}
  where $\{P^*_t\}_{t\ge0}$ is the dual semigroup of $\{P_t\}_{t\ge0}$, $W_\phi$ is the $p$-Wasserstein distance associated with the cost function $\phi$ (see the following remark for its definition).

  Now we fix an $x\in\R^n$ and take $\eta = \delta_x$. Then we have
  \begin{equation*}
    W_\phi (\mu,\eta) = W_\phi (\mu,\delta_x) = \left( \int_{\R^n} \phi(|x-y|)^p \mu(dy) \right)^{1/p}.
  \end{equation*}
  Hence,
  \begin{equation*}
    \begin{split}
      \left| P_t f(x)-\int_{\R^n} f(y)\mu(dy) \right| &\le \sup_{x \in \mathbb{R}^n}|f(x)| \cdot \|P_t(x,\cdot)-\mu\|_{\text{TV}} = \|f\|_0 \cdot \|P_t^* \delta_x -\mu\|_{\text{TV}} \\
      &\le C e^{-\rho t} W_\phi (\mu,\delta_x) = C_{x}e^{-\rho t}.
    \end{split}
  \end{equation*}
  The results follow.
\end{proof}

%\begin{remark}
%Note that, since the function $g$ is non-decreasing, the first assumption in \eqref{assump1} implies that $g(r)$ is strictly positive on
%$[r_0, \infty) $ and $\lim_{r \rightarrow \infty }g(r)=\infty$. Thus there exists a positive constant $\gamma_0$, such that
%\begin{equation}
%\langle b(x), x \rangle \leq -\gamma_0 |x|^2, ~~|x|\geq r_0 >r.
%\end{equation}
%\end{remark}

\begin{remark}
%(i). The semigroup $\{P_t\}_{t\ge0}$ preserves a finite first moment, i.e., if a measure $\eta$ has a finite first moment, then the measure $P^*_t \eta$ also has a finite first moment.
The $p$-Wasserstein distance $W_\phi$ is defined by
  \begin{equation*}
    W_\phi (\mu,\eta) = \left( \inf_{\pi\in\Pi(\mu,\eta)} \int_{\R^n\times\R^n} \phi(|x-y|)^p \pi(dx,dy) \right)^{1/p},
  \end{equation*}
  where $\Pi(\mu,\eta)$ denotes the collection of all measures on $\R^n\times\R^n$ with marginals $\mu$ and $\eta$ respectively.

\end{remark}

\subsection{\bf{Existence.}}\label{sec-3-2}

Now, we are in the position to show the existence of the solution of nonlocal Poisson equation \eqref{poisson}.
\begin{lemma}\label{wellposed-Poisson}
 Let Hypothesis ($\bf{A'_{b}}$) hold. Assume $f\in\cC_0^\mu(\R^n)$. Then the function $u$ given by \eqref{rep_u} is a solution to the equation \eqref{poisson} in $\cC_0^\mu(\R^n)$.
\end{lemma}
\begin{proof}
 We divide the proof into the following three steps.\\
{ \bf{Step 1.}} We firstly show that the right hand side of \eqref{rep_u} makes sense. Using directly the estimate \eqref{est-1} and the fact that $f$ is centered with respect to $\mu$, we have
  \begin{equation*}
    \left| \int_0^\infty \mathbb{E} f(X^x_s)ds \right| = \left| \int_0^\infty P_s f(x) ds \right| \le C_x \|f\|_0 \int_0^\infty e^{-\rho s} ds <\infty.
  \end{equation*}
  Obviously the semigroup $\{P_t\}_{t\ge0}$ is a Feller semigroup on $(\cC_0(\R^n),\|\cdot\|_0)$ \cite[Theorem 6.7.4]{Da}. The classical theory of semigroups \cite[Lemma II.1.3]{EN000} yields
  \begin{equation*}
    \L \int_0^t P_s f ds = P_t f - f.
  \end{equation*}
  { \bf{Step 2.}} Fix an $x\in\R^n$. Since $f\in\cC_0^\mu(\R^n)$, the estimate \eqref{est-1} implies that $P_t f(x)$ converges uniformly in $t$ to $0$, as $t\to\infty$. Hence, a straightforward interchange of limits yields
  \begin{equation*}
    \begin{split}
      \L u(x) &= \L \left(\lim_{t\to\infty}\int_0^t P_s f(x) ds\right) = \lim_{t\to\infty} \L \int_0^t P_s f(x) ds
      = \lim_{t\to\infty} P_t f(x) - f(x) = -f(x).
    \end{split}
  \end{equation*}
  This shows that the function $u$ defined in \eqref{rep_u} is a solution of \eqref{poisson}.\\
{ \bf{Step 3.}} We prove that $u$ is also centered with respect to $\mu$. By Fubini's theorem, we have
  \begin{equation*}
    \begin{split}
      \int_{\R^n} u(x) \mu(dx) &= \int_{\R^n} \int_0^\infty P_s f(x) ds \mu(dx) = \int_0^\infty \int_{\R^n} P_s f(x) \mu(dx)ds
      = \int_0^\infty \int_{\R^n} f(x) \mu(dx)ds = 0.
    \end{split}
  \end{equation*}
  The proof is complete now.
\end{proof}

\begin{remark}
  One can also use the Fredholm alternative to obtain the existence of solution of \eqref{poisson}, cf. \cite[Proposition 4.12]{Huang}.
\end{remark}

\section{Weak convergence in the averaging principle.}\label{sec-4}

Now we are going to apply the technique of nonlocal Poisson equations to study the weak averaging principle of the multiscale stochastic system \eqref{slo0}.
%i.e.,
%\begin{equation*}
%\label{slo}
%\left\{
%\begin{aligned}
%  d{X^{\varepsilon}_t}& =\frac{1}{\varepsilon^2}b({X^{\varepsilon}_t}, %Y^{\varepsilon}_t)dt + %\frac{1}{\varepsilon^{\frac{2}{\alpha_1}}}dL_t^{{\alpha_1}}, \;\;  %X^{\varepsilon}_0=x\in \mathbb{R}^n, \\
%  d{Y^{\varepsilon}_t}&= F({X^{\varepsilon}_t},{Y^{\varepsilon}_t})dt + %\frac{1}{\varepsilon^{r_0}}G({X^{\varepsilon}_t},{Y^{\varepsilon}_t})dt + %dL_t^{{\alpha}_2}, \;\; Y^{\varepsilon}_0=y \in \mathbb{R}^m.
%\end{aligned}
%\right.
%\end{equation*}

Introduce the following \emph{frozen equation} associated to the fast component:
\begin{equation}
\label{itro}
dX^{x,y}_t=b(X^{x,y}_t,y)dt +dL^{\alpha_1}_{t}, ~~X^{x,y}_0=x \in \mathbb{R}^n.
\end{equation}
From the discussion at the beginning of Section \ref{sec-3}, we see that the equation \eqref{itro} has a unique strong solution $X^{x,y}$ for each frozen $y \in \mathbb{R}^m$. Moreover, Lemma \ref{ergodicity0} ensures that the solution process $X^{x,y}$ possesses an invariant distribution $\mu^y$ on $\R^n$.

Motivated by Section \ref{sec-3},  we consider the following nonlocal Poisson equation
\begin{equation}
\label{elliptic}
\mathcal{L}_1\widetilde{{G}} (x,y)=-G (x,y),
\end{equation}
where
\begin{equation*}
\mathcal{L}_1\widetilde{{G}}(x,y)=-(-\Delta_x)^{\frac{\alpha_1}{2}}\widetilde{{G}}(x,y)+\left\langle b(x,y), \nabla_{x}  \widetilde{{G}}(x,y)\right\rangle.
\end{equation*}
Then Lemma \ref{wellposed-Poisson} yields that the nonlocal Poisson equation \eqref{elliptic} has a  solution $\widetilde{{G}}$, which is centered with respect to $\mu^y$ and given by
\begin{equation*}
\widetilde{{G}} (x,y)=\int^{\infty}_{0}\mathbb{E}G (X^{x,y}_t, y)dt.
\end{equation*}
%Moreover, we also have
%\begin{equation}
%\int \nabla_{y} G(x,y)\mu(dx)=0
%\end{equation}

\subsection{\bf{A priori estimates for fast  component.} }

In the following lemma, we will give some a priori estimates for the solution $X^{x,y}$ of \eqref{itro}. The proof is left in Appendix \ref{app-1}.

\begin{lemma}\label{lemma-3}
Under Hypothesis ($\bf{A_{b}}$),  for all $t\geq 0$, and $x_i \in \mathbb{R}^n$, $y_i \in \mathbb{R}^m$, $i=1,2$, we have
\begin{align}
|X^{x_1,y_1}_t-X^{x_2,y_2}_t|^2 &\leq e^{-\frac{\gamma}{2}t}|x_1-x_2|^2+C(\|b\|_1,\gamma)|y_1-y_2|^2, \label{xde} \\
|\nabla_{y}X^{x_1,y_1}_t-\nabla_{y}X^{x_2,y_2}_t|^2 &\leq C(\|b\|_2,\gamma) t e^{-\frac{\gamma}{2}t}|x_1-x_2|^2+C(\|b\|_2,\gamma)|y_1-y_2|^2, \label{xde1} \\
|\nabla_{x}X^{x_1,y_1}_t-\nabla_{x}X^{x_2,y_2}_t|^2 &\leq C(\|b\|_2,\gamma) te^{-\frac{\gamma t}{2}} \left(|y_1-y_2|^2+|x_1-x_2|^2\right), \label{xde2}
\end{align}
where $C(\|b\|_1,\gamma)$ is a constant independent of $t$.
\end{lemma}
%\begin{proof}
%\bf{See Appendix A.1.}
%\end{proof}

Now, we give the exponential ergodicity for the equation \eqref{itro}.

\begin{proposition}
\label{ergodic11}
Under Hypothesis ($\bf{A_{b}}$), for each function $\widetilde{\varphi} \in C^{1}_{b}$, there exists a positive constant $C$ such that for all $t \geq 0$ and $x \in \mathbb{R}^n$, we have
\begin{equation*}
\sup_{y \in \mathbb{R}^{m}}|P^{y}_t\widetilde{\varphi}(x)-\mu^{y}(\widetilde{\varphi})|\leq C\|\widetilde{\varphi}\|_{1}e^{-\frac{\gamma t}{8}}(1+|x|^{\frac{1}{2}}),
\end{equation*}
where
\begin{equation*}
P^{y}_{t}\widetilde{\varphi}(x)=\mathbb{E}\widetilde{\varphi}(X^{x,y}_t).
\end{equation*}
\end{proposition}
\begin{proof}
By the definition of invariant measure, \eqref{xde},  \eqref{moment0} and Lemma \ref{ergodicity0}-(i), and H\"older inequality, we have
\begin{equation*}
\begin{aligned}
\left|\mathbb{E}\widetilde{\varphi}(X^{x,y}_t)-\mu^{y}(\widetilde{\varphi})\right|
&=\left|\mathbb{E}\widetilde{\varphi}(X^{x,y}_t)-\int_{\mathbb{R}^{n}}\widetilde{\varphi}(z)\mu^{y}(dz)\right|
\leq \left|\int_{\mathbb{R}^{n}} \left[\mathbb{E}\widetilde{\varphi}(X^{x,y}_t)-\mathbb{E}\widetilde{\varphi}(X^{z,y}_t)\right]\mu^{y}(dz)\right|\\
 &\leq 2 \|\widetilde{\varphi}\|_1 \int_{\mathbb{R}^{n}} \mathbb{E}|X^{x,y}_t-X^{z,y}_t|^{\frac{1}{2}}\mu^{y}(dz)
\leq 2\|\widetilde{\varphi}\|_{1} e^{-\frac{\gamma t}{8}}\int_{\mathbb{R}^{n}}|x-z|^{\frac{1}{2}}\mu^{y}(dz)\\
& \leq 2\|\widetilde{\varphi}\|_{1} e^{-\frac{\gamma t}{8}} \left[\int_{\mathbb{R}^{n}}|x-z| \mu^{y}(dz) \right]^{\frac{1}{2}}
 \leq 2\|\widetilde{\varphi}\|_{1} e^{-\frac{\gamma t}{8}} \left[\int_{\mathbb{R}^{n}} (|x|+|z|) \mu^{y}(dz) \right]^{\frac{1}{2}}\\
& \leq C\|\widetilde{\varphi}\|_{1}e^{-\frac{\gamma t}{8}} (1+|x|^{\frac{1}{2}}).
\end{aligned}
\end{equation*}
The proof is complete.
\end{proof}

The proof of the following lemma can be found in Appendix \ref{app-1}.

\begin{lemma}\label{ps}
Under Hypotheses ($\bf{A_{b}}$),  ($\bf{A_{G1}}$) and ($\bf{A_{G2}}$), there exists a positive constant $C$ such that
\begin{align}
\sup_{y \in \mathbb{R}^m}|\widetilde{{G}}(x,y)|\leq C(1+|x|^{\frac{1}{2}}), \label{eqn-2} \\
\sup_{x\in\R^n,y\in \R^m}|\nabla_x \widetilde{{G}}(x,y)|\leq C, \label{eqn-3} \\
\sup_{y \in \mathbb{R}^m}|\nabla_y \widetilde{{G}}(x,y)|\leq C\left(1+|x|^{\frac{1}{2}}\right), \label{eqn-4} \\
\sup_{y \in \mathbb{R}^m}|\nabla^2_y \widetilde{{G}}(x,y)|\leq C(1+|x|). \label{eqn-5}
\end{align}
\end{lemma}
%\begin{proof}
%\bf{See Appendix A.4}.
%\end{proof}

We can also deduce the following moment property for fast component $X^{\varepsilon}_{t}$.
\begin{lemma}
\label{corollary1}
Let ($\bf{A_{b}}$)  hold. Then for each $1 \leq p < \alpha_1$, we have
\begin{equation*}%\label{nmoment}
\mathbb{E}\left(\sup_{0 \leq t \leq T}\left|X^{\varepsilon}_{t}\right|^p\right)\leq C_{p} \left( T^{\frac{p}{\alpha_1}}\vee T^{1-\frac{2}{\alpha_1}+\frac{p}{\alpha_1}}\right) \varepsilon^{-\frac{2p}{\alpha_1}} +|x|^p.
\end{equation*}
This implies that for any $r_1 >\frac{2p}{\alpha_1}$,
\begin{equation*}
\varepsilon^{r_1} \mathbb{E}\left(\sup_{0 \leq t \leq T}\left|X^{\varepsilon}_{t}\right|^p\right)\rightarrow 0, \quad \text{as } \varepsilon \rightarrow 0.
\end{equation*}
\end{lemma}
\begin{proof}
Note that for each $\varepsilon >0$, we have
\begin{equation*}
\begin{aligned}
X^{\varepsilon}_{t\varepsilon^2} &=x+\frac{1}{\varepsilon^2}\int_{0}^{t\varepsilon^2}b({X^{\varepsilon}_s}, Y^{\varepsilon}_s)ds + \frac{1}{\varepsilon^{\frac{2}{\alpha_1}}}L_{t\varepsilon^2}^{{\alpha_1}}
=x+\int^{t}_{0}b({X^{\varepsilon}_{s\varepsilon^2}}, Y^{\varepsilon}_{s\varepsilon^2})ds+\widetilde{L}^{\alpha_1}_{t},
\end{aligned}
\end{equation*}
where $\{\widetilde{L}^{\alpha_1}_{t}=\varepsilon^{-\frac{2}{\alpha_1}}L_{t\varepsilon^2}^{{\alpha_1}}, t \geq 0\}$ is an $\alpha$-stable process with the same law as $L^{\alpha_1}_t$.
%Define the auxilary process $\{\widetilde{X}^{\varepsilon}_t\}$  by
%\begin{equation*}
%\widetilde{X}^{\varepsilon}_t=y+\int^{t}_{0} b(\widetilde{X}^{\varepsilon}_s, %Y^{\varepsilon}_{s\varepsilon})ds +\widetilde{L}^{\alpha_1}_{t}.
%\end{equation*}
Using the condition $\sup_{y \in \mathbb{R}^m} |b(0,y)|< \infty$ and the same argument as Lemma \ref{moment0}, we can obtain that
\begin{equation*}
\mathbb{E}\left(\sup_{0 \leq t \leq T} |X^{\varepsilon}_{t\varepsilon^2} |^{p}\right) \leq C_{p}\left(T^{\frac{p}{\alpha_1}}\vee T^{1-\frac{2}{\alpha_1}+\frac{p}{\alpha_1}}\right) +|x|^p.
\end{equation*}
Therefore,
\begin{equation*}
  \begin{split}
    \mathbb{E}\left(\sup_{t \in[0,T]}|X^{\varepsilon}_t|^p\right) &= \mathbb{E}\left(\sup_{0 \leq t \leq T/{\varepsilon^2}}|X^{\varepsilon}_{t\varepsilon^2} |^{p}\right)\leq C_{p} \left( \left(\frac{T}{\varepsilon^2}\right)^{\frac{p}{\alpha_1}}\vee \left(\frac{T}{\varepsilon^2}\right)^{1-\frac{2}{\alpha_1}+\frac{p}{\alpha_1}}\right) +|x|^p \\
    &\le C_{p} \left( T^{\frac{p}{\alpha_1}}\vee T^{1-\frac{2}{\alpha_1}+\frac{p}{\alpha_1}}\right) \varepsilon^{-\frac{2p}{\alpha_1}} +|x|^p.
  \end{split}
\end{equation*}
The results follow.
%Therefore for each $1 \leq p < \alpha_1$, $r_1 >2p/{\alpha_1}$ and $T>1$, we have
%\begin{equation*}
%\varepsilon^{r_1} \mathbb{E}\left(\sup_{0 \leq t \leq %T}\left|X^{\varepsilon}_{t}\right|^p\right)\rightarrow 0, \quad \text{as } %\varepsilon \rightarrow 0.
%\end{equation*}
%The result follows.
\end{proof}

\subsection{{\bf{Tightness of the slow component}}.} \label{4.2}
In this and next subsections, we will force all hypotheses given in Section \ref{sec-2}, i.e., ($\bf{A_{b}}$), ($\bf{A_{F}}$), ($\bf{A_{G1}}$) and ($\bf{A_{G2}}$), to hold. We will use the solution of the nonlocal Poisson equation as a corrector, to show that the slow component of the original system weakly converges to the effective low dimensional system as the scale parameter tends to zero.

Given a function $f_1:\R^m\to\R$. Define
\begin{equation}
\label{fepsilon0}
f^{\varepsilon}(x,y)=f_1(y)+\varepsilon^{2-r_0} u(x,y),
\end{equation}
where we refer to $\varepsilon^{2-r_0} u$ as a \emph{corrector} to $f_1$, the function $u$ is the solution of following nonlocal Poisson equation
\begin{equation}\label{eqn-1}
\mathcal{L}_1u(x,y)=-\left\langle \nabla_y f_1(y), G(x,y)\right\rangle.
\end{equation}
Lemma \ref{wellposed-Poisson} tell us that
\begin{equation}
\label{u}
u(x,y)=\left\langle \nabla_y f_1(y),\widetilde{ {G}}(x,y)\right\rangle.
\end{equation}

By applying It\^o's formula, we have
\begin{equation}
\label{f0}
\begin{aligned}
&\ f^{\varepsilon}(X^{\varepsilon}_t, Y^{\varepsilon}_t )-f^{\varepsilon}(x, y) \\
=&\ f_1(Y^{\varepsilon}_t)-f_1(y)+\varepsilon^{2-r_0} u(X^{\varepsilon}_t, Y^{\varepsilon}_t)-\varepsilon^{2-r_0} u(x,y)\\
=&\ \int^{t}_{0} \left\langle \nabla_y f_1(Y^{\varepsilon}_s), F(X^{\varepsilon}_s, Y^{\varepsilon}_s )+ \underbrace{\varepsilon^{-r_0}}_{\text{I}} G(X^{\varepsilon}_s, Y^{\varepsilon}_s)\right\rangle ds \\
&\ + \int^{t}_{0}\left(-\left(-\Delta_y\right)^{\frac{\alpha_2}{2}} f_1(Y^{\varepsilon}_{s})\right)ds
+\int^{t}_{0}\int_{\mathbb{R}^m \backslash \{ 0\}}\left[f_1(Y^{\varepsilon}_{s-}+ y)-f_1(Y^{\varepsilon}_{s-})\right]\widetilde{N}_2(ds,dy)\\
&\ +\varepsilon^{2-r_0} \int^{t}_{0} \left\langle \nabla_y u(X^{\varepsilon}_s,Y^{\varepsilon}_s), F(X^{\varepsilon}_s, Y^{\varepsilon}_s )+\varepsilon^{-r_0}G(X^{\varepsilon}_s, Y^{\varepsilon}_s)\right\rangle ds\\
&\ + \varepsilon^{2-r_0}  \int^{t}_{0}\left(-\left(-\Delta_y\right)^{\frac{\alpha_2}{2}}u(X^\varepsilon_s, Y^{\varepsilon}_s)\right)ds
+\varepsilon^{2-r_0} \int^{t}_{0}\int_{\mathbb{R}^m \backslash \{ 0\}}\left[u(X^\varepsilon_{s-},Y^{\varepsilon}_{s-}+ y)-u(X^\varepsilon_{s-},Y^{\varepsilon}_{s-})\right]\widetilde{N}_2(ds,dy)\\
&\ + \underbrace{\varepsilon^{-r_0}}_{\text{II}} \int^{t}_{0} \mathcal{L}_1u(X^\varepsilon_s,Y^{\varepsilon}_s)ds +\varepsilon^{2-r_0}\int^{t}_{0}\int_{\mathbb{R}^n \backslash \{ 0\}}\left[u(X^\varepsilon_{s-}+ \varepsilon^{-\frac{2}{\alpha_1}}x,Y^{\varepsilon}_{s-})- u(X^\varepsilon_{s-},Y^{\varepsilon}_{s-})\right]\widetilde{N}_1(ds,dx),
\end{aligned}
\end{equation}
where
\begin{equation*}
-\left(-\Delta_y\right)^{\frac{\alpha_2}{2}} f_1(y)=\int_{\mathbb{R}^m \backslash \{ 0\}} \left[ f_1(y+z)-f_1(y)-I_{\{|z|\leq 1\}}\langle z, \partial_{y}f_1(y)\rangle \right] \nu_{2}(dz).
\end{equation*}

Since $u$ solves the equation \eqref{eqn-1}, the two terms of order $\varepsilon^{-r_0}$ in I and II in \eqref{f0} cancel out. Thus we obtain
\begin{equation}
\label{fff0}
\begin{aligned}
f_1(Y^{\varepsilon}_t)=&\ f_1(y) +\int^{t}_{0} \left\langle \nabla_y f_1(Y^{\varepsilon}_s), F(X^{\varepsilon}_s, Y^{\varepsilon}_s )+\varepsilon^{2-2r_0}\sum_{i}G_i(X^\varepsilon_s,Y^{\varepsilon}_s)\partial_{y_i}\widetilde{{G}}(X^\varepsilon_s,Y^{\varepsilon}_s)\right\rangle ds \\
&\ +\varepsilon^{2-2r_0}\int^{t}_0 \sum_{i,j}\partial_{y_i}\partial_{y_j}f_1(Y^{\varepsilon}_s)G_i{\widetilde{G}}_j(X^\varepsilon_s,Y^{\varepsilon}_s)ds
+\int^{t}_{0}\left(-\left(-\Delta_y\right)^{\frac{\alpha_2}{2}} f_1(Y^{\varepsilon}_{s})\right)ds\\
&\ + \int^{t}_{0}\int_{\mathbb{R}^m\backslash \{ 0\}}\left[f_1(Y^{\varepsilon}_{s-}+ y)-f_1(Y^{\varepsilon}_{s-})\right]\widetilde{N}_2(ds,dy)+\varepsilon^{2-r_0} R^{\varepsilon}_u(t),
\end{aligned}
\end{equation}
where
\begin{equation}
\label{R}
\begin{aligned}
R^{\varepsilon}_u(t)=&\ u(x,y)-u(X^{\varepsilon}_t, Y^{\varepsilon}_t)+\int^{t}_{0} \left\langle \nabla_y u(X^{\varepsilon}_s,Y^{\varepsilon}_s), F(X^{\varepsilon}_s, Y^{\varepsilon}_s )\right\rangle ds\\
&\ + \int^{t}_{0}\left(-\left(-\Delta_y\right)^{\frac{\alpha_2}{2}}u(X^\varepsilon_s, Y^{\varepsilon}_s)\right)ds
+ \int^{t}_{0}\int_{\mathbb{R}^m \backslash \{ 0\}}\left[u(X^\varepsilon_{s-},Y^{\varepsilon}_{s-}+ y)-u(X^\varepsilon_{s-},Y^{\varepsilon}_{s-})\right]\widetilde{N}_2(ds,dy)\\
&\ +\int^{t}_{0}\int_{\mathbb{R}^m \backslash \{ 0\}}\left[u\left(X^\varepsilon_{s-}+ \varepsilon^{-\frac{2}{\alpha_1}}x,Y^{\varepsilon}_{s-}\right)- u(X^\varepsilon_{s-},Y^{\varepsilon}_{s-})\right]\widetilde{N}_1(ds,dy).
\end{aligned}
\end{equation}

In the following, we will show the relative compactness of $\{Y^{\varepsilon}, \varepsilon>0\}$ in the metric space $\mathbb{D}\left([0,T], \mathbb{R}^{m}\right)$.
\begin{lemma}
\label{pro}
The family $\{Y^{\varepsilon}, \varepsilon>0\}$ of second solution processes of \eqref{slo0} satisfies the following conditions:\\
(i) For all $T>0$ and $\delta>0$, there exists $N>0$, such that
\begin{equation*}
\mathbb{P}\left(\sup_{0 \leq t \leq T}|Y^{\varepsilon}_t|>N\right)\leq \delta, ~~\forall 0 <\varepsilon <1;
\end{equation*}
(ii) For all $T>0$, it holds that
\begin{equation*}
\begin{aligned}
\lim_{\delta\rightarrow 0}\limsup_{\varepsilon\to0}\sup_{0\le\tau< T-\delta}\mathbb{P}(|Y^{\varepsilon}_{\tau+\delta}-Y^{\varepsilon}_{\tau}|>\lambda)=0, ~~\forall \lambda >0,
\end{aligned}
\end{equation*}
where the second supremum is taken over all stopping time $\tau$ satisfying $0\le\tau< T-\delta$.
Consequently, \cite[Theorem VI.4.5]{JS13} yields that the family $\{Y^{\varepsilon}, \varepsilon>0\}$ is relative compact.
\end{lemma}
\begin{proof}
(i). Choose in \eqref{fepsilon0} that $f_1(y)=\log(1+|y|^2)$. Then we get
\begin{equation}
\label{f}
(1+|y|)\cdot|\nabla_{y} f_1(y)|+(1+|y|)^2\cdot|\nabla^{2}_{y}f_1(y)|+(1+|y|)^3\cdot|\nabla^3_yf_1(y)|\leq C.
\end{equation}
By Lemma \ref{ps} and \eqref{u}, we have for every $x \in \R^n,y \in \R^m$,
\begin{equation*}%\label{ues}
|u(x,y)|\leq C(1+|x|^{\frac{1}{2}}),
\end{equation*}
and
\begin{equation}
\label{deria}
|\nabla_{x}u(x,y)| \leq C, \quad |\nabla_{y}u(x,y)| \leq C(1+|x|^{\frac{1}{2}}), \quad |\nabla^2_{y}u(x,y)| \leq C(1+|x|^{\frac{1}{2}})+C(1+|x|).
\end{equation}
Using the above estimates and Lemma \ref{corollary1}, we obtain
\begin{align}
\mathbb{E}_{x,y} \left( \sup_{0 \leq t \leq T}|u(X^{\varepsilon}_t, Y^{\varepsilon}_t )| \right) &\leq C  \mathbb{E}_{x,y} \left( \sup_{0 \leq t \leq T}(1+|X^{\varepsilon}_t|^{\frac{1}{2}}) \right) < \infty, \label{est-8} \\
\mathbb{E}_{x,y} \left( \sup_{0 \leq t \leq T}|\nabla_{y}u(X^{\varepsilon}_t, Y^{\varepsilon}_t )|^2 \right) & \leq C  \mathbb{E}_{x,y} \left( \sup_{0 \leq t \leq T}(1+|X^{\varepsilon}_t|) \right) < \infty, \label{ub} \\
\mathbb{E}_{x,y} \left( \sup_{0 \leq t \leq T}|\nabla^2_{y}u(X^{\varepsilon}_t, Y^{\varepsilon}_t )| \right) & \leq C  \mathbb{E}_{x,y} \left( \sup_{0 \leq t \leq T}(1+|X^{\varepsilon}_t|) \right) < \infty. \label{est-7}
\end{align}
Recall from \eqref{fff0} and \eqref{R} that
\begin{equation}
\label{f0f}
\begin{aligned}
f_1(Y^{\varepsilon}_t) - f_1(Y^{\varepsilon}_{0})=&\ \int^{t}_{0} \left\langle \nabla_y f_1(Y^{\varepsilon}_s), F(X^{\varepsilon}_s, Y^{\varepsilon}_s )+\varepsilon^{2-2r_0}\sum_{i}G_i(X^\varepsilon_s,Y^{\varepsilon}_s)\partial_{y_i}\widetilde{{G}}(X^\varepsilon_s,Y^{\varepsilon}_s)\right\rangle ds \\
&\ +\varepsilon^{2-2r_0}\int^{t}_0 \sum_{i,j}\partial_{y_i}\partial_{y_j}f_1(Y^{\varepsilon}_s)G_i{\widetilde{G}}_j(X^\varepsilon_s,Y^{\varepsilon}_s)ds
+\int^{t}_{0}\left(-\left(-\Delta_y\right)^{\frac{\alpha_2}{2}} f_1(Y^{\varepsilon}_{s})\right)ds\\
&\ + \int^{t}_{0}\int_{\mathbb{R}^m\backslash \{ 0\}}\left[f_1(Y^{\varepsilon}_{s-}+ y)-f_1(Y^{\varepsilon}_{s-})\right]\widetilde{N}_2(ds,dy)+\varepsilon^{2-r_0}\left[u(X^{\varepsilon}_{0},Y^{\varepsilon}_{0})-u(X^{\varepsilon}_t, Y^{\varepsilon}_t)\right]\\
&\ +\varepsilon^{2-r_0}\left[\int^{t}_{0} \left\langle \nabla_y u(X^{\varepsilon}_s,Y^{\varepsilon}_s), F(X^{\varepsilon}_s, Y^{\varepsilon}_s )\right\rangle ds\right]+\varepsilon^{2-r_0}\left[\int^{t}_{0}\left(-\left(-\Delta_y\right)^{\frac{\alpha_2}{2}}u(X^\varepsilon_s, Y^{\varepsilon}_s)\right)ds\right]\\
&\ +\varepsilon^{2-r_0}\left[\int^{t}_{0}\int_{\mathbb{R}^m \backslash \{ 0\}}\left[u(X^\varepsilon_{s-},Y^{\varepsilon}_{s-}+ y)-u(X^\varepsilon_{s-},Y^{\varepsilon}_{s-})\right]\widetilde{N}_2(ds,dy)\right]\\
&\ +\varepsilon^{2-r_0}\left[\int^{t}_{0}\int_{\mathbb{R}^m \backslash \{ 0\}}\left[u\left(X^\varepsilon_{s-}+ \varepsilon^{-\frac{2}{\alpha_1}}x,Y^{\varepsilon}_{s-}\right)- u(X^\varepsilon_{s-},Y^{\varepsilon}_{s-})\right]\widetilde{N}_1(ds,dx)\right]\\
=:&\ I_1+I_2+I_3+I_4+J_1+J_2+J_3+J_4+J_5.
\end{aligned}
\end{equation}
In the following, we will estimate the terms in \eqref{f0f} one-by-one.

For the term $I_1$, we use the inequality \eqref{f}, Hypothesis ($\bf{A_{F}}$), Hypothesis ($\bf{A_{G1}}$), \eqref{index} and \eqref{eqn-4} to deduce that
\begin{equation}
\label{I_1}
\begin{aligned}
\mathbb{E}\left[\sup_{0 \leq t \leq T}|I_1|\right] & \leq C_{T}+\varepsilon^{2-2r_0}\mathbb{E}\left[\int^{T}_{0}\left(1+|X^{\varepsilon}_t|^{\frac{1}{2}}\right)\right]
 \leq C_T \int_0^T\mathbb{E}\left[\left(1+|X^{\varepsilon}_{s-}|\right)\right] ds.
\end{aligned}
\end{equation}
For the same reason, we also have the following estimate for the term $I_2$,
\begin{equation*}
\label{I_2}
\mathbb{E}\left[\sup_{0 \leq t \leq T}|I_2|\right] \leq C_T \int_0^T \mathbb{E}\left[\left(1+|X^{\varepsilon}_{s-}|\right)\right] ds.
\end{equation*}
For the term $I_3$, by the choice of $f_1$ and \eqref{f}, we have
\begin{equation*}
\label{I_3}
\begin{aligned}
&\mathbb{E}\left[\sup_{0 \leq t \leq T}\left|\int^{t}_0\left(-(-\Delta_y)^{\frac{\alpha_2}{2}}f_1(Y^{\varepsilon}_{s-})\right)ds\right|\right]\\
&=\mathbb{E}\left[\sup_{0 \leq t \leq T}\left|\int^{t}_{0}\left\{\int_{\mathbb{R}^m}\left[f_1(Y^{\varepsilon}_{s-}+z)-f_1(Y^{\varepsilon}_{s-})-
I_{\{|z|\leq 1\}}\left\langle z, \nabla_{y} f_1(Y^{\varepsilon}_{s-})\right\rangle \right]v_{2}(dz)\right\}ds\right|\right]\\
& \leq \int^T_{0}\left\{\|\nabla_{y} f_1\|_{0}\int_{|z|>1}z\nu_2(dz)+C_{T}\|\nabla^2_{y} f_1\|_0\int_{|z| \leq 1}z^2\nu_2(dz)\right\}ds\\
& \leq C_{T}.
\end{aligned}
\end{equation*}
For the term $I_4$,  by the Burkholder--Davis--Gundy inequality \cite[Theorem 3.50]{PZ07},  Jensen's inequality, the proof of \cite[Lemma 8.22]{PZ07} and \eqref{f}, we have
\begin{equation}\label{I_4}
\begin{aligned}
\mathbb{E}\left(\sup_{0 \leq t \leq T}|I_4|\right)&\leq \mathbb{E}\left[\sup_{t\in [0,T]}\left|\int^{t}_{0}\int_{\mathbb{R}^m \backslash \{ 0\}} \left[f_1(Y^{\varepsilon}_{s-}+ y)-f_1(Y^{\varepsilon}_{s-})\right]\widetilde{N}_2(ds,dy)\right|\right]\\
%&\quad +\mathbb{E}\left[\sup_{t\in [0,T]}\left|\int^{t}_{0}\int_{|y|\geq 1}\left[f_1(Y^{\varepsilon}_{s-}+ y)-f_1(Y^{\varepsilon}_{s-})\right]\widetilde{N}_2(ds,dy)\right|\right]\\
& \leq  \mathbb{E}\left[\left|\int^{T}_{0}\int_{|y|<1}\left[f_1(Y^{\varepsilon}_{s-}+ y)-f_1(Y^{\varepsilon}_{s-})\right]^2{N}_2(ds,dy)\right|\right]^{\frac{1}{2}}\\
&\quad + \mathbb{E}\left[\left|\int^{T}_{0}\int_{|y|\geq 1}\left[f_1(Y^{\varepsilon}_{s-}+ y)-f_1(Y^{\varepsilon}_{s-})\right]^2{N}_2(ds,dy)\right|\right]^{\frac{1}{2}}\\
&\leq \left\{\int^{T}_{0}\int_{|y|<1}\mathbb{E}
\left[f_1(Y^{\varepsilon}_{s-}+y)-f_1(Y^{\varepsilon}_{s-})\right]^2
\nu_{2}(dy)ds\right\}^{\frac{1}{2}} \\
&\quad +\left\{\int^{T}_{0}\int_{|y|\geq 1}\mathbb{E}\left| f_1(Y^{\varepsilon}_{s-}+y)-f_1(Y^{\varepsilon}_{s-})\right|
\nu_{2}(dy)ds\right\} \\
& \leq C\left[\int^{T}_{0}\int_{|y|<1}|y|^2\nu_{2}(dy)ds\right]^{\frac{1}{2}}+C\left[\int^{T}_{0}\int_{|y|\geq 1}|y|\nu_{2}(dy)ds\right]\\
&\leq  C_T.
\end{aligned}
\end{equation}
For the term $J_1$, we use \eqref{est-8} to have
\begin{equation*}
\label{J_1}
\mathbb{E}\left[\sup_{0 \leq t \leq T}|J_1|\right]\leq C \varepsilon^{2-r_0}\left\{1+\mathbb{E}\left[\sup_{0 \leq t \leq T}|X^{\varepsilon}_t|\right]\right\}^{1/2}.
\end{equation*}
For the term $J_2$, by \eqref{ub}, H\"older inequality and Hypothesis $(\bf{A_{F}})$, we obtain
\begin{equation*}
\label{J_2}
\mathbb{E}\left[\sup_{0 \leq t \leq T}|J_2|\right]\leq \varepsilon^{2-r_0}\mathbb{E}\left[\left(1+\sup_{0 \leq t \leq T}|X^{\varepsilon}_t|\right)\right].
\end{equation*}
For the term $J_3$, we apply the estimates \eqref{ub} and \eqref{est-7} to obtain
\begin{equation*}
\label{J_3}
\begin{aligned}
&\quad\ \mathbb{E}\left[\sup_{0 \leq t \leq T}|J_3|\right]\\
&=\mathbb{E}\left[\sup_{0 \leq t \leq T}\left[\varepsilon^{2-r_0}\int^{t}_{0}\left\{\int_{\mathbb{R}^m}\left[u(X^{\varepsilon}_{s-},Y^{\varepsilon}_{s-}+z)-u(X^{\varepsilon}_{s-},Y^{\varepsilon}_{s-})-
I_{\{|z|\leq 1\}}\langle z, \nabla_{y} u(X^{\varepsilon}_{s-},Y^{\varepsilon}_{s-})\rangle \right]v_{2}(dz)\right\}ds\right]\right]\\
& \leq \varepsilon^{2-r_0}\int^T_{0}\mathbb{E}\left\{\|\nabla_y u(X^{\varepsilon}_{s-},\cdot)\|_{0}\int_{|z|>1}z\nu_2(dz)+ \|\nabla^2_{y} u(X^{\varepsilon}_{s-},\cdot)\|_{0}\int_{|z| \leq 1}z^2\nu_2(dz)\right\}ds\\
& \leq \varepsilon^{2-r_0}C\left({\|u\|_2,\alpha_2}\right) \int^T_{0} \left(1+\mathbb{E}\left[\sup_{0 \leq t \leq T}|X^{\varepsilon}_{s-}|\right]\right) ds.
\end{aligned}
\end{equation*}
For the term $J_4$, we use the same argument as the estimates of $I_4$ in \eqref{I_4}, as well as the second inequality of \eqref{deria}. Then we have
\begin{equation*}
\label{J_4}
\begin{aligned}
\mathbb{E}\left(\sup_{0 \leq t \leq T}|J_4|\right)&\leq \varepsilon^{2-r_0}\mathbb{E}\left[\sup_{t\in [0,T]}\left|\int^{t}_{0}\int_{\mathbb{R}^m \backslash \{ 0\}} \left[u(X^{\varepsilon}_{s-},Y^{\varepsilon}_{s-}+ y)-u(X^{\varepsilon}_{s-},Y^{\varepsilon}_{s-})\right]\widetilde{N}_2(ds,dy)\right|\right]\\
%&+\varepsilon^{2-r_0}\mathbb{E}\left[\sup_{t\in [0,T]}\left|\int^{t}_{0}\int_{|y|\geq 1}\left[u(X^{\varepsilon}_{s-},Y^{\varepsilon}_{s-}+ y)-u(X^{\varepsilon}_{s-},Y^{\varepsilon}_{s-})\right]\widetilde{N}_2(ds,dy)\right|\right]\\
& \leq \varepsilon^{2-r_0} \mathbb{E}\left[\int^{T}_{0}\int_{|y|< 1}\left[u(X^{\varepsilon}_{s-},Y^{\varepsilon}_{s-}+ y)-u(X^{\varepsilon}_{s-},Y^{\varepsilon}_{s-})\right]^2{N}_2(ds,dy)\right]^{\frac{1}{2}}\\
&\quad +\varepsilon^{2-r_0} \mathbb{E}\left[\int^{T}_{0}\int_{|y|\geq 1}\left[u(X^{\varepsilon}_{s-},Y^{\varepsilon}_{s-}+ y)-u(X^{\varepsilon}_{s-},Y^{\varepsilon}_{s-})\right]^2{N}_2(ds,dy)\right]^{\frac{1}{2}}\\
&\leq \varepsilon^{2-r_0}\left\{\int^{T}_{0}\int_{|y|<1}\mathbb{E}
\left[u(X^{\varepsilon}_{s-},Y^{\varepsilon}_{s-}+y)-u(X^{\varepsilon}_{s-},Y^{\varepsilon}_{s-})\right]^2
\nu_{2}(dy)ds\right\}^{\frac{1}{2}}\\
&\quad +\varepsilon^{2-r_0}\left\{\int^{T}_{0}\int_{|y|\geq 1}\mathbb{E} \left|u(X^{\varepsilon}_{s-},Y^{\varepsilon}_{s-}+y)- u(X^{\varepsilon}_{s-},Y^{\varepsilon}_{s-})\right|
\nu_{2}(dy)ds\right\} \\
& \leq C\varepsilon^{2-r_0}\left[\mathbb{E} \int^{T}_{0}\int_{|y|<1} |y|^2\nu_{2}(dy)(1+|X^{\varepsilon}_{s-}|)ds\right]^{\frac{1}{2}} \\
&\quad +C\varepsilon^{2-r_0}\left[\mathbb{E}\int^{T}_{0}\int_{|y|\geq 1}|y|\nu_{2}(dy)(1+|X^{\varepsilon}_{s-}|^{\frac{1}{2}})ds\right]\\
&\leq C_T \varepsilon^{2-r_0} \mathbb{E}\left(1+\sup_{0 \leq t \leq T}\left|X^{\varepsilon}_{t}\right|\right).
\end{aligned}
\end{equation*}
Similarly, we also have the following estimates for $J_5$,
\begin{equation}
\label{J_5}
\begin{aligned}
\mathbb{E}\left(\sup_{0 \leq t \leq T}|J_5|\right) \leq &\ \varepsilon^{2-r_0}\mathbb{E}\left[\sup_{t\in [0,T]}\left|\int^{t}_{0}\int_{\mathbb{R}^m \backslash \{ 0\}}\left[u(X^{\varepsilon}_{s-}+\varepsilon^{-\frac{2}{\alpha_1}}x,Y^{\varepsilon}_{s-})
-u(X^{\varepsilon}_{s-},Y^{\varepsilon}_{s-})\right]\widetilde{N}_1(ds,dx)\right|\right]\\
%&\ +\varepsilon^{2-r_0}\mathbb{E}\left[\sup_{t\in [0,T]}\left|\int^{t}_{0}\int_{|x|\geq 1}\left[u(X^{\varepsilon}_{s-}+\varepsilon^{-\frac{2}{\alpha_1}}x,Y^{\varepsilon}_{s-})-u(X^{\varepsilon}_{s-},Y^{\varepsilon}_{s-})\right]\widetilde{N}_1(ds,dx)\right|\right]\\
\leq &\ \varepsilon^{2-r_0} \mathbb{E}\left[\int^{T}_{0}\int_{|x|<1}\left[u(X^{\varepsilon}_{s-}+\varepsilon^{-\frac{2}{\alpha_1}}x,Y^{\varepsilon}_{s-})-u(X^{\varepsilon}_{s-},Y^{\varepsilon}_{s-})\right]^2{N}_1(ds,dx)\right]^{\frac{1}{2}}\\
&\ +\varepsilon^{2-r_0} \mathbb{E}\left[\left|\int^{T}_{0}\int_{|x|\geq 1}\left[u(X^{\varepsilon}_{s-}+\varepsilon^{-\frac{2}{\alpha_1}}x,Y^{\varepsilon}_{s-})-u(X^{\varepsilon}_{s-},Y^{\varepsilon}_{s-})\right]^2{N}_1(ds,dx)\right|\right]^{\frac{1}{2}}\\
\leq&\ \varepsilon^{2-r_0}\left\{\int^{T}_{0}\int_{|x|<1}\mathbb{E}
\left[u(X^{\varepsilon}_{s-}+\varepsilon^{-\frac{2}{\alpha_1}}x,Y^{\varepsilon}_{s-})-u(X^{\varepsilon}_{s-},Y^{\varepsilon}_{s-})\right]^2
\nu_{1}(dx)ds\right\}^{\frac{1}{2}}\\
&\ +\varepsilon^{2-r_0}\left\{\int^{T}_{0}\int_{|x|\geq 1}\mathbb{E} \left| u(X^{\varepsilon}_{s-}+\varepsilon^{-\frac{2}{\alpha_1}}x,Y^{\varepsilon}_{s-}) -u(X^{\varepsilon}_{s-},Y^{\varepsilon}_{s-})\right|
\nu_{1}(dx)ds\right\} \\
\leq &\ C\varepsilon^{2-r_0-2/{\alpha_1}} \left[\int^{T}_{0}\int_{|x|<1}|x|^2\nu_{1}(dx)ds\right]^{\frac{1}{2}} +C\varepsilon^{2-r_0-2/{\alpha_1}} \int^{T}_{0}\int_{|x|\geq 1}|x|\nu_{1}(dx)ds \\
\leq&\ C_T \varepsilon^{2-r_0-2/{\alpha_1}}.
\end{aligned}
\end{equation}

Plugging the above estimates  \eqref{I_1}--\eqref{J_5} into \eqref{f0f} and using the condition \eqref{index}, we obtain
\begin{equation}
\label{log}
\sup_{0< \varepsilon \leq 1}\mathbb{E}_{x,y}\left(\sup_{0 \leq t \leq T}\log \left(1+|Y^{\varepsilon}_t|^2\right)\right)< \infty.
\end{equation}
In view of \eqref{log}, we use Chebyshev's inequality to get
\begin{equation*}
\begin{aligned}
\mathbb{P}\left({\sup_{0 \leq t \leq T}}|Y^{\varepsilon}_t|>N\right)&= \mathbb{P}\left(\log\left(1+{\sup_{0 \leq t \leq T}}|Y^{\varepsilon}_t|^2\right)>\log(1+N^2)\right) \\
& \leq \frac{\sup_{0< \varepsilon \leq 1}\mathbb{E}_{x,y}\left(\sup_{0 \leq t \leq T}\log \left(1+|Y^{\varepsilon}_t|^2\right)\right)}{\log(1+N^2)}\\
&\rightarrow 0, ~~ N\rightarrow \infty.
\end{aligned}
\end{equation*}

(ii). Let $\tau \leq T-\delta_0$ be a bounded stopping time. For any $\delta \in (0, \delta_0)$, by the strong Markov property, we have
\begin{equation}
\label{418}
\mathbb{P}\left(|Y^{\varepsilon}_{\tau+\delta}-Y^{\varepsilon}_{\tau}|>\lambda\right) =\mathbb{E}\left(\mathbb{P} \left(|Y^{\varepsilon}_{s+\delta}-y|>\lambda\right)|_{(s,y)=(\tau, Y^{\varepsilon}_{\tau})}\right).
\end{equation}
Define
\begin{equation*}
\widetilde{Y}^{\varepsilon}_t=Y^{\varepsilon}_t-y,
\end{equation*}
then we have
\begin{equation*}
\left\{
\begin{aligned}
d{\widetilde{Y}^{\varepsilon}_t}&= F({X^{\varepsilon}_t},{\widetilde{Y}^{\varepsilon}_t}+y)dt + \frac{1}{\varepsilon}G({X^{\varepsilon}_t},{\widetilde{Y}^{\varepsilon}_t+y})dt + dL_t^{{\alpha}_2}\\
 \widetilde{Y}^{\varepsilon}_0&=0 \in \mathbb{R}^m.
 \end{aligned}
 \right.
\end{equation*}

Let us write \eqref{f0f} in the particular case of the vector function $f_1(y)=y$. We obtain
\begin{equation*}%\label{f0f1}
\begin{aligned}
\widetilde{Y}^{\varepsilon}_t =&\ \widetilde{Y}^{\varepsilon}_{0} +\int^{t}_{0} \left\langle \mathbf{I}, F(X^{\varepsilon}_s, \widetilde{Y}^{\varepsilon}_s +y)+\varepsilon^{2-2r_0}\sum_{i}G_i(X^\varepsilon_s,\widetilde{Y}^{\varepsilon}_s+y)\partial_{y_i}\widetilde{{G}}(X^\varepsilon_s,\widetilde{Y}^{\varepsilon}_s)\right\rangle ds +\int^{t}_{0}\left[-\left(-\Delta_y\right)^{\frac{\alpha_2}{2}} \widetilde{Y}^{\varepsilon}_{s}\right]ds\\
&\ + \int^{t}_{0}\int_{\mathbb{R}^m\backslash \{0\}}y\widetilde{N}_2(ds,dy)+\varepsilon^{2-r_0}\left[u(X^{\varepsilon}_{0},\widetilde{Y}^{\varepsilon}_{0})-u(X^{\varepsilon}_t, \widetilde{Y}^{\varepsilon}_t)\right]\\
&\ +\varepsilon^{2-r_0}\left[\int^{t}_{0} \left\langle \nabla_y u(X^{\varepsilon}_s,\widetilde{Y}^{\varepsilon}_s), F(X^{\varepsilon}_s, \widetilde{Y}^{\varepsilon}_s )\right\rangle ds\right]+\varepsilon^{2-r_0}\left[\int^{t}_{0}\left(-\left(-\Delta_y\right)^{\frac{\alpha_2}{2}}u(X^\varepsilon_s, \widetilde{Y}^{\varepsilon}_s)\right)ds\right]\\
&\ +\varepsilon^{2-r_0}\left[\int^{t}_{0}\int_{\mathbb{R}^m \backslash \{ 0\}}\left[u(X^\varepsilon_{s-},\widetilde{Y}^{\varepsilon}_{s-}+ y)-u(X^\varepsilon_{s-},Y^{\varepsilon}_{s-})\right]\widetilde{N}_2(ds,dy)\right]\\
&\ +\varepsilon^{2-r_0}\left[\int^{t}_{0}\int_{\mathbb{R}^m \backslash \{ 0\}}\left[u\left(X^\varepsilon_{s-}+ \varepsilon^{-\frac{2}{\alpha_1}}x,\widetilde{Y}^{\varepsilon}_{s-}\right)- u(X^\varepsilon_{s-},\widetilde{Y}^{\varepsilon}_{s-})\right]\widetilde{N}_1(ds,dx)\right].
\end{aligned}
\end{equation*}

Using a technique similar to (i), we obtain
\begin{equation}
\label{425}
\sup_{0< \varepsilon \leq 1}\mathbb{E}_{x,y}\sup_{0 \leq t \leq T}|\widetilde{Y}^{\varepsilon}_t|<C T^{\frac{1}{\alpha_1}}.
\end{equation}
Thus we have
\begin{equation*}
\begin{aligned}
\mathbb{P}_{s,y}\left(|\widetilde{Y}^{\varepsilon}_{s+\delta}|>\lambda\right)
& \leq \frac{\mathbb{E}^{\mathbb{P}_{s,y}}\left(|\widetilde{Y}^{\varepsilon}_{s+\delta}| \right)}{\lambda} \leq  \frac{C \delta^{\frac{1}{\alpha_1}}}{\lambda}.
\end{aligned}
\end{equation*}
Combining \eqref{418} and \eqref{425}, we obtain
\begin{equation*}
\begin{aligned}
\mathbb{P}\left(|Y^{\varepsilon}_{\tau+\delta}-Y^{\varepsilon}_{\tau}|>\lambda\right) &\leq \mathbb{P}(|Y^\varepsilon_{\tau}|>R)+ \frac{C \delta^{\frac{1}{\alpha_1}}}{\lambda}
 \leq \frac{C T^{\frac{1}{\alpha_1}}}{R}+ \frac{C \delta^{\frac{1}{\alpha_1}}}{\lambda}.
\end{aligned}
\end{equation*}
Letting $\delta\rightarrow 0$ first and then $R\rightarrow \infty$, one sees that (ii) is satisfied.

\end{proof}

\subsection{\bf{Proof of Theorem 2.}}

%Let $\Lambda$ be the set of injective increasing functions $\lambda$ satisfying $\lim_{t\rightarrow - \infty}\lambda(t)=- \infty$ and $\lim_{t\rightarrow \infty}\lambda(t)=\infty$. Define the following metric
%\begin{equation*}
%d_{\mathbb{R}}(\omega_1, \omega_2)=\inf \left\{\varepsilon \textgreater 0:\sup_{t\in \mathbb{R}}\left|\omega_1(t)-\omega_2(\lambda (t))\right|\leq \varepsilon, \sup_{s\neq t}\left|\ln \frac{\arctan(\lambda (t))-\arctan(\lambda (s))}{\arctan(t)-\arctan(s)}\right|\leq \varepsilon\right\}.
%\end{equation*}
%Then we have the following well-known result, see e.g. \cite[Theorem 1]{pl}.

Consider the Skorokhod space $\mathbb{D}\left([0,T], \mathbb{R}^{m}\right)$ consisting of all $\mathbb{R}^{m}$-valued c\`adl\`ag functions on $[0,T]$, equipped with the Skorokhod topology. It is well-known that $\mathbb{D}\left([0,T], \mathbb{R}^{m}\right)$ is a Polish space (e.g., \cite[Section VI.1]{JS13} or \cite[Section 14]{Bil68}).

%\begin{lemma}
%\label{separate}
% The space $\mathbb{D}\left([0,T], \mathbb{R}^{m}\right)$ equipped with the distance $d_{\mathbb{R}}$ is a Polish space.
%\end{lemma}

Next, we will present the uniform approximation of c\`adl\`ag functions by step functions, which comes from \cite[Lemma 9, Appendix A]{rc}.
\begin{lemma}
\label{stepa}
Let $h$ be a c\`adl\`ag function on $[0,T]$. If $(t^n_k)$ is a sequence of subdivisions $0=t^n_0< t_1 <\cdots < t^{n}_{k_n}=t$ of  $ [0, T ]$ such that
\begin{equation*}
\sup_{0 \leq i \leq k-1}|t^{n}_{i+1}-t^{n}_{i}|\rightarrow 0, ~~\sup_{u \in [0, T]\backslash \{ t^n_0, \cdots, t^n_{k_n}\}}|\Delta h(u)|\rightarrow 0, ~~ as ~~n\rightarrow \infty.
\end{equation*}
Then we have
\begin{equation*}
\sup_{u \in [0, T]}\left|h(u)-\sum_{0}^{k_n -1}h(t_i)I_{t^{n}_{i}, t^{n}_{i+1}(u)+h(t^{n}_{k_n})I_{t^n_{k_n}}}(u)\right|\rightarrow 0.
\end{equation*}
\end{lemma}

Due to the tightness of the family $\{Y^{\varepsilon}, \varepsilon>0\}$, there exists a subsequence $\varepsilon_{n}\rightarrow 0$ and a stochastic process $Y$, such that $Y^{\varepsilon_{n}}$ converges weakly to $Y$, as $n\rightarrow \infty$. Based on Lemma \ref{stepa}, we have the following result.
\begin{lemma}
\label{appo}
For sufficient small $\delta >0$, there exist a positive constant $N\in \mathbb{N}$ and $\mathbb{R}^{m}$-valued step functions $y^1, y^2, \cdots, y^N$ such that
\begin{equation*}
\begin{aligned}
&\mathbb{P}\left(\bigcap_{k=1}^N \left\{d_{\mathbb{R}}(Y^{\varepsilon_n},y^{k})>\delta\right\}\right)< \delta, ~~ \forall n \in \mathbb{N}, \\
&\mathbb{P}\left(\bigcap_{k=1}^N \left\{d_{\mathbb{R}}(Y,y^{k})>\delta\right\}\right)< \delta.
\end{aligned}
\end{equation*}
\end{lemma}
\begin{proof}
{ \bf{Step 1.}} By the tightness of the set $\{Y, Y^{\varepsilon_n}, n \in \mathbb{N}\}$, there exists a compact set $K\subseteq \mathbb{D}([0,T]; \mathbb{R}^{m})$ such that for each $0<\delta<<1$, we have
\begin{equation*}%\label{Y}
\begin{aligned}
&\mathbb{P}\left(Y^{\varepsilon_n}\in K\right)>1-\delta, ~~
\mathbb{P}\left(Y\in K\right)>1-\delta.
\end{aligned}
\end{equation*}
Since $K$ is compact, it is totally bounded. Hence, $K$ admits a finite $\delta/2$-net, i.e., there exists a finite subset $\{\widetilde{y}^1, \widetilde{y}^2, \cdots, \widetilde{y}^N \}\subseteq\mathbb{D}([0,T]; \mathbb{R}^{m})$ s.t.
%\begin{equation}
%\bigcup_{k=1}^{\infty} \left\{x\in \mathbb{D}([0,T]; \mathbb{R}^{m}) : d_{\mathbb{R}}(\widetilde{y}^{k},x)<\frac{\delta}{2}\right\}\supseteq \mathbb{D}\left([0,T]; \mathbb{R}^{m}\right)\supseteq K.
%\end{equation}
%Because $K$ is a compact set, there exists a positive integer $N$  such that
\begin{equation*}
K \subseteq \bigcup_{k=1}^{N} \left\{x\in \mathbb{D}([0,T]; \mathbb{R}^{m}) : d_{\mathbb{R}}(\widetilde{y}^{k},x)<\frac{\delta}{2}\right\}.
\end{equation*}
Set $A_k=\left\{x\in \mathbb{D}([0,T]; \mathbb{R}^{m}) : d_{\mathbb{R}}(\widetilde{y}^{k},x)<\frac{\delta}{2}\right\}$, then we have
\begin{equation*}%\label{A}
\bigcap_{k=1}^{N}A^{c}_k \subseteq K^{c}.
\end{equation*}
This implies
\begin{equation*}
\begin{aligned}
&\mathbb{P}\left(\bigcap_{k=1}^N \left\{d_{\mathbb{R}}(Y^{\varepsilon_n},\widetilde{y}^{k})>\frac{\delta}{2}\right\}\right)< \frac{\delta}{2}, ~~ \forall n \in \mathbb{N}, \\
&\mathbb{P}\left(\bigcap_{k=1}^N \left\{d_{\mathbb{R}}(Y,\widetilde{y}^{k})>\frac{\delta}{2}\right\}\right)< \frac{\delta}{2}.
\end{aligned}
\end{equation*}
{ \bf{Step 2.}} By Lemma \ref{stepa}, for fixed $t\geq 0$, we can find the step function $y^k$, which is arbitrarily close to the c\`adl\`ag function  in supremum norm, i.e.,
\begin{equation*}
\sup_{0 \leq t \leq T}|y^k-\widetilde{y}^k|<\frac{\delta}{2}.
\end{equation*}
Then we have
\begin{equation*}
\begin{aligned}
\left\{d_{\mathbb{R}}(Y^{\varepsilon_n}, y^k)>\delta\right\}& \subseteq
\left\{d_{\mathbb{R}}(Y^{\varepsilon_n}, \widetilde{y}^k)>\frac{\delta}{2}\right\}\bigcup  \left\{d_{\mathbb{R}}(y^{k},
 \widetilde{y}^k)>\frac{\delta}{2}\right\}\\
 & \subseteq
\left\{d_{\mathbb{R}}(Y^{\varepsilon_n}, \widetilde{y}^k)>\frac{\delta}{2}\right\}\bigcup  \left\{\sup|y^{k}-
 \widetilde{y}^k)|>\frac{\delta}{2}\right\},
 \end{aligned}
\end{equation*}
and
\begin{equation*}
\begin{aligned}
\left\{d_{\mathbb{R}}(Y, y^k)>\delta\right\}& \subseteq
\left\{d_{\mathbb{R}}(Y, \widetilde{y}^k)>\frac{\delta}{2}\right\}\bigcup  \left\{d_{\mathbb{R}}(y^{k},
 \widetilde{y}^k)>\frac{\delta}{2}\right\}
  \subseteq \left\{d_{\mathbb{R}}(Y, \widetilde{y}^k)>\frac{\delta}{2}\right\}\bigcup  \left\{\sup|y^{k}-
 \widetilde{y}^k)|>\frac{\delta}{2}\right\}.
 \end{aligned}
\end{equation*}
Therefore we have
\begin{equation*}
\begin{aligned}
&\mathbb{P}\left(\bigcap_{k=1}^N \left\{d_{\mathbb{R}}(Y^{\varepsilon_n},y^{k})>\delta\right\}\right)< \delta, ~~ \forall n \in \mathbb{N}, \\
&\mathbb{P}\left(\bigcap_{k=1}^N \left\{d_{\mathbb{R}}(Y,y^{k})>\delta\right\}\right)< \delta.
\end{aligned}
\end{equation*}

\end{proof}

In what follows, we will fix a $\delta>0$ and let $y^1, y^2, \cdots, y^N$ be the corresponding $N$ step functions in Lemma \ref{appo}. For each $y \in \mathbb{D}([0,T]; \mathbb{R}^{m})$ and $k=1,2,\cdots,N$, we define
\begin{equation*}
\beta_{k}(y):=d_{\mathbb{R}}(y,y^{k}).
\end{equation*}
Let $\psi, \varphi_1, \cdots, \varphi_N : \mathbb{D}\left(\left[0,T\right]; \mathbb{R}^{m}\right)\rightarrow [0,1]$ be smooth mappings such that
\begin{enumerate}
\item[(i)]$\psi(y)+\sum_{k=1}^n\varphi_k(y)=1$, $\forall y \in \mathbb{D}\left([0,T]; \mathbb{R}^{m}\right)$;
\item[(ii)]$\supp \psi \subset \bigcap_{k=1}^N \left\{y; \beta_{k}(y)>\delta\right\}$;
\item[(iii)] $\supp \varphi_k \subset \bigcap_{k=1}^N \{y; \beta_{k}(y)<2\delta\}$, $1\leq k \leq N$.
\end{enumerate}
We shall introduce the following $[0,1]$-valued random variables:
\begin{equation}
\label{nota}
\begin{aligned}
\xi_n&:=\psi(Y^{\varepsilon_n}), ~~~~\xi:=\psi(Y),\\
 \eta^k_n&:=\varphi_k(Y^{\varepsilon_n}), ~~~~\eta^k:=\varphi_k(Y),
 \end{aligned}
\end{equation}
and two measurable sets
\begin{equation*}
\begin{aligned}
\widetilde{A}_{n}&:=\bigcap^N_{k=1}\left\{\omega ; d_{\mathbb{R}}(Y^{\varepsilon_n}(\omega), y^{k})>\delta \right\}, ~~
 \widetilde{A}&:=\bigcap^N_{k=1}\left\{\omega; d_{\mathbb{R}}(Y(\omega), y^{k})>\delta \right\}.
 \end{aligned}
\end{equation*}
In view of Lemma \ref{appo}, we clearly have
\begin{equation*}
\supp\xi_n\subseteq \widetilde{A}_{n},~~~\supp\xi \subseteq \widetilde{A},
\end{equation*}
where for a real-valued random variable $\zeta$, its support $\supp \zeta $ is defined by
 \begin{equation*}
 \supp \zeta:=\left\{\omega \in \Omega; \zeta(\omega)\neq 0 \right\}.
 \end{equation*}
Set
\begin{equation*}%\label{rew1}
\begin{aligned}
\Lambda_{n}(t)&:= \Gamma_{n}(t) - {\varepsilon_{n}}^{2-r_0} R^{\varepsilon}_u(t_0,t) \Phi_{t_{0}}\left(Y^{\varepsilon_{n}}\right), \\
\Gamma_{n}(t)&:=\left[\phi(Y^{\varepsilon_{n}}_t)-\phi(Y^{\varepsilon_{n}}_{t_{0}})-\int^{t}_{t_0} \left\langle \nabla_y \phi(Y^{\varepsilon_{n}}_s), F(X^{\varepsilon_{n}}_s,Y^{\varepsilon_{n}}_s)\right\rangle ds-\int^{t}_{t_0}\left(-\left(-\Delta_y\right)^{\frac{\alpha}{2}}\phi(Y^{\varepsilon_{n}}_s)\right)ds\right]\Phi_{t_0}(Y^{\varepsilon_{n}}),
\end{aligned}
\end{equation*}
where  $\Phi_{{t_0}}(\cdot)$ is a bounded function on $\mathbb{D}\left(\left[0,T\right]\right)$, which is measurable with respect to the sigma-field $\sigma\left(\omega_{t}, \omega \in \mathbb{D}\left(\left[0,T\right]  \right), 0 \leq t\leq t_0\right)$, $\phi$ is a $C^{\infty}_{0}$ function on $\mathbb{R}^{m}$.

\begin{remark}
Here we introduce $[0, 1]$-valued random variables $\xi_n, ~\xi, ~\eta^{n}_{k}, ~\eta^{k}$ to prove the $L^{1}(\Omega)$-convergence of $\Gamma_{n}(t)$ and give the explicit expression of convergence result for $\Gamma_{n}(t)$.
\end{remark}
For convenience, we  assume that $X^\varepsilon_{t_0}=\mathbf{x}, Y^\varepsilon_{t_0}=\mathbf{y}$.   It follows from the similar arguments used in the proof of tightness in Lemma \ref{pro}, we have
\begin{equation}
\label{r0}
\lim_{{\varepsilon}_{n}\to 0}{\varepsilon_{n}}^{2-r_0} \mathbb{E}_{\mathbf{x},\mathbf{y}}\left[R^{\varepsilon}_{u}(t_0,t)\Phi_{t_{0}}(Y^{\varepsilon})\right]=0,
\end{equation}
and
\begin{equation*}\label{r1}
\lim_{{\varepsilon}_{n}\to 0}{\varepsilon_{n}}^{2-r_0}\sum_{i=1}^{m} \mathbb{E}_{\mathbf{x},\mathbf{y}}\left[\int^{t}_{t_{0}}\left\langle \nabla_{y}\phi(Y^{\varepsilon_{n}}_{s}),G_i(X^{\varepsilon_{n}}_{s}, Y^{\varepsilon_{n}}_{s})\partial_{y_{i}} \widetilde{G}(X^{\varepsilon_{n}}_{s}, Y^{\varepsilon_{n}}_{s})\right \rangle ds \Phi_{t_{0}}(Y^{\varepsilon})\right]=0.
\end{equation*}
 Since the integral $\int^{t}_{0}\int_{\mathbb{R}^m \backslash \{ 0\}} \cdots\widetilde{N}_2(ds,dy)$ in $\mathbb{R}^m$ is a martingale with respect to the $\sigma$-algebras $\mathcal{F}_{t}$ generated by $\left\{ \widetilde{N}_2(s,dy); s\leq t\right\}$, we have
\begin{equation}
\label{r2}
\mathbb{E}_{\mathbf{x},\mathbf{y}}\left[\int^{t}_{t_0}\int_{\mathbb{R}^m\backslash \{ 0\}}\left[\phi(Y^{\varepsilon}_{s-}+ y)-\phi(Y^{\varepsilon}_{s-})\right]\widetilde{N}_2(ds,dy)\Phi_{t_{0}}(Y^{\varepsilon})\right]=0.
\end{equation}

In view of \eqref{fff0} and \eqref{r0}--\eqref{r2},  we gain
\begin{equation}
\label{0Gamma}
\mathbb{E}_{\mathbf{x},\mathbf{y}}\left[\Gamma_n(t)\right] \to 0, ~~n \to \infty.
\end{equation}
Define $\Gamma^{k}_{n}$ as the random variable $\Gamma_{n}$, where $Y^{\varepsilon_{n}}$ is replaced by $y^{k}$. Then we have the following result.
\begin{lemma}
\label{lipc}
For every $\delta >0$  and each positive integral $k=1,2, \cdots, N$, there exists a positive integer $M_0$, such that $|Y^{\varepsilon_n}_t-y^{k}_t|\leq \delta$, we have
\begin{equation*}
\sum_{k=1}^N \mathbb{E}\left[(\Gamma_n(t)-\Gamma^k_n(t))\eta^k_n(t)\right]\leq M_0 \delta.
\end{equation*}
\end{lemma}
\begin{proof}
By the definition of the function $\psi$ and $\varphi_k, k=1,2, \cdots, N$ in \eqref{nota} and the equation \eqref{0Gamma}, we have
\begin{equation*}
\mathbb{E}_{\mathbf{x},\mathbf{y}}(\Gamma_{n}\xi_n)+\sum_{k=1}^{N}\mathbb{E}_{\mathbf{x},\mathbf{y}}(\Gamma_{n}\eta^{k}_n)\to 0,  n \to \infty.
\end{equation*}

For each positive integral $k=1,2, \cdots, N$, we have
\begin{equation*}
\begin{aligned}
&\ \mathbb{E}\left[\left(\Gamma_n(t)-\Gamma^k_n(t)\right)\eta^k_{n}(t)\right] \\
=&\ \left\{\mathbb{E}\left[\phi(Y^{\varepsilon_{n}}_{t})
-\phi(Y^{\varepsilon_{n}}_{t_{0}}) \right]\Phi_{t_{0}}(Y^{\varepsilon_{n}})
-\mathbb{E}\left[\phi(y^{k}_{t})-\phi(y^{k}_{t_0})\right]\Phi_{t_{0}}(y^{k})\right\}\\
&\ -\left\{\mathbb{E}\left[\int^{t}_{t_0}\left\langle \nabla_y \phi(Y^{\varepsilon_n}_s), F(X^{\varepsilon_n}_s, Y^{\varepsilon_n}_s)\right\rangle ds
\right]\Phi_{t_{0}}(Y^{\varepsilon_n})
-\mathbb{E}\left[\int^{t}_{t_0}\left\langle \nabla_y \phi(y^{k}_s), F(X^{\varepsilon_n}_s, y^{k}_s)\right\rangle ds\right]\Phi_{t_{0}}(y^{k})\right\} \\
&\ -\left\{\mathbb{E}\left[\int^{t}_{t_{0}}\left(-(-\Delta_y)^{\frac{\alpha_{2}}{2}}\phi(Y^{\varepsilon_n}_s)\right)ds\right]\Phi_{t_{0}}
(Y^{\varepsilon_{n}})
-\mathbb{E}\left[\int^{t}_{t_0}\left(-(-\Delta_y)^{\frac{\alpha_{2}}{2}}\phi(y^{k}_s)\right)ds\right]\Phi_{t_{0}}(y^{k})\right\}.
\end{aligned}
\end{equation*}
By $\phi\in C^{\infty}_{0}$, Hypothesis ($\bf{A_{F}}$) and the boundedness of $\Phi_{t_{0}}(\cdot)$, we get the required result.
\end{proof}

Set
\begin{equation*}
\Gamma(t)=\left[\phi(Y_t)-\phi(Y_{t_{0}})-\int^{t}_{t_0} \left\langle \nabla_y \phi(Y_s), {\bar{F}}(Y_s)\right\rangle ds-\int^{t}_{t_0}\left(-\left(-\Delta_y\right)^{\frac{\alpha}{2}}\phi(Y_s)\right)ds\right]\Phi_{t_0}(Y),
\end{equation*}
where the function ${\bar{F}}(y)$ is defined by
\begin{equation*}
\bar{F}(y)=\int_{\mathbb{R}^n} F(x,y)\mu^{y}(dx).
\end{equation*}
Define $\Gamma^k$ as the quantity obtained by replacing $Y$ by $y^k$ in the expression for $\Gamma$. By the same technique as Lemma \ref{lipc}, we also get the following corollary.
\begin{corollary}
 For every $\delta >0$, there exists a positive integer $\widehat{{M}}$, such that
\begin{equation*}
\sum_{k=1}^N \mathbb{E}\left[(\Gamma(t)-\Gamma^k(t))\eta^k(t)\right]\leq \widehat{M} \delta.
\end{equation*}
\end{corollary}
%Before proving the required convergence, we need to give an auxiliary proposition and some related lemmas.

Introduce a new auxiliary process $\widetilde{X}^{\varepsilon, x,y}_t$, which satisfies the following stochastic differential equation
\begin{equation}
\label{sde23}
d\widetilde{X}^{\varepsilon,x, y}_t=b\left({\widetilde{X}}^{\varepsilon, x, y}_t,y\right)dt+d\widetilde{L}^{\varepsilon, \alpha_1}_t, ~~\widetilde{X}^{\varepsilon,x,y}_0=x,
\end{equation}
where $\widetilde{L}^{\varepsilon, \alpha_1}_t$ is defined by
\begin{equation*}
\widetilde{L}^{\varepsilon,\alpha_1}_t=\frac{1}{\varepsilon^{\frac{2}{\alpha_1}}}L^{\alpha_1}_{\varepsilon^2 t}.
\end{equation*}
Obviously, the process $\widetilde{L}^{\varepsilon,\alpha_1}_t$ is also an $\alpha$-stable process, with the same law as $L^{\alpha_1}_t$.

%In the following, we will give the prior estimate for the jump process $\widetilde{X}^{\varepsilon, x,y}_t$.
%
%
%\begin{lemma}
%\label{xde12}
%Under Hypotheses ($\bf{A_{b}}$),  for all $t\geq 0$, and $x_i \in \mathbb{R}^n$, $y_i \in \mathbb{R}^m$, $i=1,2$, we have \begin{equation}
%|\widetilde{X}^{\varepsilon, x_1,y_1}_t-\widetilde{X}^{\varepsilon, x_2,y_2}_t|^2 \leq e^{-\frac{\gamma}{2}t}|x_1-x_2|^2+C(\|b\|_1,\gamma)|y_1-y_2|^2,
%\end{equation}
%where $C(\|b\|_1,\gamma)$ is a constant independent of $t$.
%\end{lemma}
%\begin{proof}
%The proof is similar to Lemma \ref{xde}.
%\end{proof}

Next, we will give the exponential ergodicity for the equation \eqref{sde23}. Using the similar technique as \eqref{xde} and Proposition \ref{ergodic11},  we have

\begin{proposition}%\label{ergodic12}
Under Hypothesis ($\bf{A_{b}}$), for each function $\widetilde{\varphi} \in C^{1}_{b}$, there exists a positive constant $C$ such that for all $t \geq 0$ and $x \in \mathbb{R}^n$, we have
\begin{equation*}
\sup_{y \in \mathbb{R}^{m}}|P^{\varepsilon, x,y}_t\widetilde{\varphi}(x)-\mu^{y}(\widetilde{\varphi})|\leq C\|\widetilde{\varphi}\|_{1}e^{-\frac{\gamma t}{4}}(1+|x|^{\frac{1}{2}}),
\end{equation*}
where
\begin{equation*}
P^{\varepsilon,x,y}_{t}\widetilde{\varphi}(x)=\mathbb{E}\widetilde{\varphi}(\widetilde{X}^{\varepsilon, x,y}_t),
\end{equation*}
and the positive constant $C$ is independent of $\varepsilon$.
\end{proposition}
%\begin{proof}
%By the definition of invariant measure and  Lemma \ref{xde12},  we have
%\begin{equation}
%\begin{aligned}
%\left|\mathbb{E}\widetilde{\varphi}(\widetilde{X}^{\varepsilon,x,y}_t)-\mu^{y}(\widetilde{\varphi})\right|
%&=\left|\mathbb{E}\widetilde{\varphi}(\widetilde{X}^{\varepsilon,x,y}_t)-\int_{\mathbb{R}^{n}}\widetilde{\varphi}(z)\mu^{y}(dz)\right|\\
%&\leq \left|\int_{\mathbb{R}^{n}} \left[\mathbb{E}\widetilde{\varphi}(\widetilde{X}^{\varepsilon,x,y}_t)
%-\mathbb{E}\widetilde{\varphi}(\widetilde{X}^{\varepsilon,z,y}_t)\right]\mu^{y}(dz)\right|\\
%& \leq 2 \left(|\widetilde{\varphi}|_{0}+\|\nabla\widetilde{\varphi}\|_{0}\right)\int_{\mathbb{R}^{n}} \mathbb{E}|\widetilde{X}^{\varepsilon,x,y}_t-\widetilde{X}^{\varepsilon,z,y}_t|^{\frac{1}{2}}\mu^{y}(dz) \\
%& \leq C\|\widetilde{\varphi}\|_{1} e^{-\frac{\gamma t}{8}}\int_{\mathbb{R}^{n}}|x-z|^{\frac{1}{2}}\mu^{y}(dz)\\
%& \leq C\|\widetilde{\varphi}\|_{1}e^{-\frac{\gamma t}{8}}(1+|x|^{\frac{1}{2}}).
%\end{aligned}
%\end{equation}
%\end{proof}

Now, we are in the position to prove the $L^{1}(\Omega)$-convergence of $\Gamma^{k}_{n}(t)$.

\begin{lemma}%\label{Gamma}
Let $K\in \mathcal{C}^{1,0}_{b}$ and $\bar{K}(y):=\int_{\mathbb{R}^n} K(x,y)\mu^{y}(dx)$, then for every $0<t<T$, we have
\begin{equation*}
\mathbb{E}\left|\int ^t_0\left(K(X^{\varepsilon_n}_s, y^k_s)-\bar{K}(y^k_s)\right)ds\right|\rightarrow 0, ~~as~~ n\rightarrow \infty.
\end{equation*}
\end{lemma}
\begin{proof}
Let $(a_k, b_k)\subseteq [0,T]$ be an interval on which $y^k_{\cdot}$ is a constant, denoted by $z^k$. This can be done by Lemma \ref{stepa}. Then we will only show
\begin{equation*}
\mathbb{E}\left|\int ^{b_k}_{a_k}\left[K(X^{\varepsilon_n}_s, z^k)-\bar{K}(z^k)\right]ds\right|\rightarrow 0, ~~as~~ n\rightarrow \infty.
\end{equation*}
By the equation \eqref{slo0}, we know that
\begin{equation*}
\begin{aligned}
X^{\varepsilon_n}_{t\varepsilon^2_n}&=x+\int^{t}_{0}b(X^{\varepsilon_n}_{u\varepsilon^2_n}, Y^{\varepsilon_n}_{u\varepsilon^2_n})du+\frac{1}{\varepsilon_n^{\frac{2}{\alpha_1}}}L^{\alpha_1}_{t\varepsilon^2_n}, \\
Y^{\varepsilon_n}_{t\varepsilon^2_n}&=y+\varepsilon_n^2\int^{t}_{0}F(X^{\varepsilon_n}_{u\varepsilon^2_n}, Y^{\varepsilon_n}_{u\varepsilon^2_n})du+\varepsilon^{2-r_0}_n\int^{t}_{0}G(X^{\varepsilon_n}_{u\varepsilon^2_n}, Y^{\varepsilon_n}_{u\varepsilon^2_n})du+L^{\alpha_2}_{t\varepsilon^2_n}.
\end{aligned}
\end{equation*}
Then by Hypotheses ($\bf{A_{F}}$),($\bf{A_{G1}}$) and  the inequality (2.8) in \cite{czhang},  we have
\begin{equation}
\label{ylabel}
\mathbb{E}\left|Y^{\varepsilon_n}_{t\varepsilon^2_n }-y\right| \to 0, ~~ n \to \infty.
\end{equation}
%On the one hand, by the self-similarity  of $\alpha$-stable process, we have
%\begin{equation}
%\label{481}
%\mathbb{E}\left|\frac{1}{\varepsilon_n^{\frac{2}{\alpha_1}}}L^{\alpha_1}_{t\varepsilon^2_n}\right|=\mathbb{E}\left|L^{\alpha_1}_t\right|< \infty.
%\end{equation}
Therefore we have
\begin{equation*}
\begin{aligned}
\mathbb{E}\left|X^{\varepsilon_n}_{t\varepsilon^2_n}-\widetilde{X}^{\varepsilon,x, y}_t\right|&=\mathbb{E}\left|\int^{t}_{0}\left[b(X^{\varepsilon_n}_{u\varepsilon^2_n}, Y^{\varepsilon_n}_{u\varepsilon^2_n})\right]du-\int^{t}_{0}\left[b(\widetilde{X}^{\varepsilon,x, y}_u, y)\right]du\right|\\
& \leq \|\nabla_x b\|_{0}\int^t_{0}\mathbb{E}|X^{\varepsilon_n}_{\varepsilon^2_nu}-\widetilde{X}^{\varepsilon,x, y}_u|ds +\|\nabla_y b\|_{0}\int^t_{0}\mathbb{E}|Y^{\varepsilon_n}_{u\varepsilon^2_n}-y|du.
\end{aligned}
\end{equation*}
By Gr\"onwall's inequality and \eqref{ylabel}, we have
\begin{equation*}
\mathbb{E}\left|X^{\varepsilon_n}_{t\varepsilon^2_n}-\widetilde{X}^{\varepsilon,x,y}_t\right| \to 0, ~~n \to \infty.
\end{equation*}
On the other hand, by $K\in \mathcal{C}^{1,0}_{b}$ and Proposition \ref{ergodic11}, we know

\begin{equation*}
\begin{aligned}
&\ \mathbb{E}\left|\int^{b_k}_{a_k}\left[K(X^{\varepsilon_n}_s, z^k)-\bar{K}(z^k)\right]ds\right| = \varepsilon^2_n\mathbb{E}\left|\int^{b_k/{\varepsilon^2_n}}_{a_k/{\varepsilon^2_n}}\left[K(X^{\varepsilon_n}_{s{\varepsilon^2_n}}, z^k)-\bar{K}(z^k)\right]ds\right|\\
\leq &\ \varepsilon^2_n\mathbb{E}\left|\int^{b_k/{\varepsilon^2_n}}_{a_k/{\varepsilon^2_n}}
\left[K(X^{\varepsilon_n}_{s{\varepsilon^2_n}}, z^k)-K(\widetilde{X}^{\varepsilon,x,z^k}_s, z^k)\right]ds\right| \\
&\ +\varepsilon^2_n\left|\int^{b_k/{\varepsilon^2_n}}_{a_k/{\varepsilon^2_n}}\mathbb{E}\left( K(\widetilde{X}^{\varepsilon,x,z^k}_s, z^k)\right)ds-\int^{b_k/{\varepsilon^2_n}}_{a_k/{\varepsilon^2_n}}\int_{\mathbb{R}^n}K(x,z^k)\mu^{z^k}(dx)ds\right|\\
\leq &\ \varepsilon^2_n \|\nabla_x k\|_{0}\mathbb{E}\left[\int^{b_k/{\varepsilon^2_n}}_{a_k/{\varepsilon^2_n}}
|X^{\varepsilon_n}_{s{\varepsilon^2_n}}-\widetilde{X}^{\varepsilon,x,z^k}_s| ds\right]+\varepsilon^2_n\int^{b_k/{\varepsilon^2_n}}_{a_k/{\varepsilon^2_n}}C\|K\|_{1,0}e^{-\gamma s}\left(1+|x|^{\frac{1}{2}}\right)ds \\
\to&\ 0, ~~n \to \infty.
\end{aligned}
\end{equation*}

\end{proof}
\subsubsection{\bf{Proof of Theorem 1.}}
Now, we are in the position to give :

\noindent Proof of Theorem 1. We divide the proof into the following two steps.\\
{ \bf{Step 1.}} By Lemma \ref{pro}, there exists a subsequence $\{\varepsilon_n\}\rightarrow 0$, such that $Y^{\varepsilon_n} $ converges weakly to a limit point $Y$, as $n\rightarrow \infty$.

For any $p, q>1$ and $\frac{1}{p}+\frac{1}{q}=1$, by H\"older inequality, Lemma \ref{appo} and $\phi\in C^{\infty}_{0}$, we have
\begin{equation}
\label{steg}
\begin{aligned}
|\mathbb{E}\left(\Gamma_n(t)\xi_n(t)\right)|&\leq \left(\mathbb{E}(\Gamma_n(t))^p\right)^{1/p}\left(\mathbb{P}(\widetilde{A}_n(t))\right)^{1/q}\leq C\delta^{1/q}, \\
|\mathbb{E}\left(\Gamma(t)\xi(t)\right)|&\leq \left(\mathbb{E}(\Gamma(t))^p\right)^{1/p}\left(\mathbb{P}(\widetilde{A}(t))\right)^{1/q}\leq C\delta^{1/q}.
\end{aligned}
\end{equation}
Since  $\eta^k_n\Rightarrow \eta^k$ and $\eta^k_n\leq 1$. Thus for each $k=1,2, \cdots, N$, we have
\begin{equation}
\begin{aligned}
\label{res1}
\left|\mathbb{E}\left(\Gamma^n_k\eta^k_n-\Gamma^k\eta^k\right)\right|&\leq \mathbb{E}\left|\left(\Gamma^n_k-\Gamma^k\right)\eta^k_n\right|+|\eta^k|\left|\mathbb{E}\left(\eta^k_n-\eta^k\right)\right|
 \leq \mathbb{E}\left|\Gamma^n_k-\Gamma^k\right|+\mathbb{E}\left|\eta^k_n-\eta^k\right|\times \Gamma^k.
\end{aligned}
\end{equation}
{ \bf{Step 2.}} By the definitions of $\psi$ and $\varphi$, we yield
\begin{equation*}
\begin{aligned}
&\mathbb{E}\left(\Gamma_n \xi_{n}(t)\right)+\sum_{k=1}^{N}\mathbb{E}_{x,y}\left(\Gamma_n \eta^{k}_{n}(t)\right)\rightarrow 0, ~~as~~ n\rightarrow \infty,~~
\mathbb{E}(\Gamma)=\mathbb{E}(\Gamma\xi_t)+\sum_{k=1}^{N}\mathbb{E}_{x,y}\left(\Gamma\eta^{k}_t\right),
\end{aligned}
\end{equation*}
and
\begin{equation}
\label{sep}
\begin{aligned}
\sum_{k=1}^{N}\mathbb{E}(\Gamma_n\eta^{k}_{n})&=\sum_{k=1}^N\mathbb{E}\left[(\Gamma_n-\Gamma^k_n)\eta^n_k\right]+\sum_{k=1}^N \mathbb{E}(\Gamma^k_n\eta^k_n), \\
\sum_{k=1}^{N}\mathbb{E}(\Gamma\eta^{k})&=\sum_{k=1}^N\mathbb{E}\left[(\Gamma-\Gamma^k)\eta_k\right]
+\sum_{k=1}^N \mathbb{E}(\Gamma^k\eta^k).
\end{aligned}
\end{equation}
By Lemma \ref{lipc}, \eqref{steg}, \eqref{res1} and \eqref{sep}, we get
\begin{equation}
\label{conver}
\mathbb{E}[\Gamma_n]\rightarrow \mathbb{E}[\Gamma], ~~ \text{as}~~ n \rightarrow \infty.
\end{equation}
Combining \eqref{0Gamma} and \eqref{conver}, we have
\begin{equation*}
\mathbb{E}\left[\left(\phi(Y_t)-\phi(Y_{t_0})-\int^{t}_{t_0}\mathcal{L}_2\phi(Y_s)ds\right)\Phi_{t_{0}}(Y_u)\right]=0,
\end{equation*}
where
\begin{equation*}
\mathcal{L}_2\phi(y) = \left\langle \nabla_y \phi(y), {\bar{F}}(y)\right\rangle -\left(-\Delta_y\right)^{\frac{\alpha}{2}}\phi(y).
\end{equation*}

\section{Concluding remarks.}\label{sec-5}

In this paper, we study the weak averaging principle for multiscale systems
driven by $\alpha$-stable noises. First, we examine the existence of the nonlocal Poisson
equation corresponding to an ergodic jump
process. Then by constructing suitable correctors, we obtain the tightness of the slow component. It turns out that the slow component weakly converges to a jump process as the scale parameter $\varepsilon$  goes to zero.

There are some limitations to this paper. Firstly, the condition $1<\alpha_i <2, ~~i=1,2$ plays an important role in deriving the effective dynamical system. How to obtain the effective low dimensional system and to estimate the effects that the fast components have on slow ones are still open for the case $\alpha_i\in (0,1)$. Secondly, the slow components contain homogenization terms, whose homogenizing index $r_0$ has a relation with the stable index $\alpha_1$ of the noise of fast components given by $0<r_0<1-1/{\alpha_1}$. How to relax the restriction for the homogenizing index is also an active issue.
Thirdly, the above multiscale stochastic dynamical systems are driven by additive stable L\'evy noises.
It is also interesting to consider the effective dynamics of multiscale stochastic dynamical systems driven by multiplicative noises. Finally, it is worth emphasizing that the estimate \eqref{est-1} depends on the variable $x$, so the classical semigroup method is no longer suitable. Therefore, for such a case, it is necessary to find some new approaches to study. In future works, we will examine the uniqueness and regularity of nonlocal elliptic equations under the assumption of exponential ergodicity, as well as the central limit theorem for multiscale systems with homogenization terms. Further, we will also study the dependence of the convergence rate on the regularity of the coefficients of the slow component.

\section{Appendix A. Further Proofs.}\label{app-1}

\subsection{Proof of Lemma \ref{lemma-3}.}

\begin{proof}[Proof of \eqref{xde}]

By the equation \eqref{itro}, we have
\begin{equation*}
d\left(X^{x_1,y_1}_t-X^{x_2,y_2}_t\right)=\left[b(X^{x_1,y_1}_t,y_1)-b(X^{x_2,y_2}_t,y_2)\right] dt, ~~X^{x_1,y_1}_0-X^{x_2,y_2}_0=x_1-x_2.
\end{equation*}
Multiplying both sides by $2\left(X^{x_1,y_1}_t-X^{x_2,y_2}_t\right)$, by Assumption ($\bf{A_{b}}$) and Young's inequality, we have
\begin{equation*}
\begin{aligned}
\frac{d}{dt}\left|X^{x_1,y_1}_t-X^{x_2,y_2}_t\right|^2&=2\left\langle b(X^{x_1,y_1}_t, y_1)-b(X^{x_2,y_2}_t, y_2), X^{x_1,y_1}_t-X^{x_2,y_2}_t \right\rangle \\
& \leq 2\left\langle b(X^{x_1,y_1}_t, y_1)-b(X^{x_2,y_2}_t, y_1), X^{x_1,y_1}_t-X^{x_2,y_2}_t \right\rangle \\
&\quad +2\left\langle b(X^{x_2,y_2}_t, y_1)-b(X^{x_2,y_2}_t, y_2), X^{x_1,y_1}_t-X^{x_2,y_2}_t \right\rangle \\
& \leq -2\gamma\left|X^{x_1,y_1}_t-X^{x_2,y_2}_t\right|^2+C(\|\nabla_y b\|_0, \gamma)|y_1-y_2|\left|X^{x_1,y_1}_t-X^{x_2,y_2}_t\right|\\
& \leq - \gamma \left|X^{x_1,y_1}_t-X^{x_2,y_2}_t\right|^2+C(\|\nabla_y b\|_0,\gamma)|y_1-y_2|^2.
\end{aligned}
\end{equation*}
Hence, the comparison theorem yields that
\begin{equation*}
\left|X^{x_1,y_1}_t-X^{x_2,y_2}_t\right|^2 \leq e^{-\frac{\gamma}{2}t}|x_1-x_2|^2 +C(\|b\|_1,\gamma)|y_1-y_2|^2.
\end{equation*}

\end{proof}

\begin{proof}[Proof of \eqref{xde1}]

Note that
\begin{equation*}
d\left(\nabla_{y}X^{x,y}_t\right)=(\nabla_{x}b)(X^{x,y}_t,y)\nabla_{y}X^{x,y}_tdt +(\nabla_{y}b)(X^{x,y}_t,y)dt, ~~\nabla_{y}X^{x,y}_0=\bf{0}.
\end{equation*}
This implies that
\begin{equation*}
\begin{aligned}
d(\nabla_{y}X^{x_1,y_1}_t-\nabla_{y}X^{x_2,y_2}_t)= &\left((\nabla_{x}b)(X^{x_1,y_1}_t,y_1)\nabla_{y}X^{x_1,y_1}_t-(\nabla_{x}b)(X^{x_2,y_2}_t,y_2)\nabla_{y}X^{x_2,y_2}_t\right)dt\\ &+((\nabla_{y}b)(X^{x_1,y_1}_t,y_1)-(\nabla_{y}b)(X^{x_2,y_2}_t,y_2))dt.
\end{aligned}
\end{equation*}
Multiplying both sides by $2(\nabla_{y}X^{x_1,y_1}_t-\nabla_{y}X^{x_2,y_2}_t)$, we have
\begin{equation*}
\begin{aligned}
&\ \frac{d}{dt}|\nabla_{y}X^{x_1,y_1}_t-\nabla_{y}X^{x_2,y_2}_t|^2
\\ =&\ 2\langle (\nabla_{x}b)(X^{x_1,y_1}_t,y_1)\nabla_{y}X^{x_1,y_1}_t- (\nabla_{x}b)(X^{x_2,y_2}_t,y_2)\nabla_{y}X^{x_2,y_2}_t, \nabla_{y}X^{x_1,y_1}_t-\nabla_{y}X^{x_2,y_2}_t\rangle
\\&+ 2\langle (\nabla_{y}b)(X^{x_1,y_1}_t,y_1)-(\nabla_{y}b)(X^{x_2,y_2}_t,y_2), \nabla_{y}X^{x_1,y_1}_t-\nabla_{y}X^{x_2,y_2}_t\rangle
\\=:&\ \Sigma_1+\Sigma_2.
\end{aligned}
\end{equation*}
For the term $\Sigma_1$, observe that by substituting $x_2=x_1+\varepsilon h$ and letting $\varepsilon \rightarrow 0$ in the dissipative condition \eqref{0disb}, we have for all $x, h \in \mathbb{R}^{n}, y\in \mathbb{R}^{m}$,
\begin{equation}
\label{disv}
\langle \nabla_{x}b(x,y)h, h\rangle \leq -\gamma |h|^2.
\end{equation}
Therefore, we have
\begin{equation*}
\begin{aligned}
\Sigma_1
& \leq 2\left\langle (\nabla_{x}b)(X^{x_1,y_1}_t,y_1)\nabla_{y}X^{x_1,y_1}_t- (\nabla_{x}b)(X^{x_1,y_1}_t,y_1)\nabla_{y}X^{x_2,y_2}_t ,\nabla_{y}X^{x_1,y_1}_t-\nabla_{y}X^{x_2,y_2}\right \rangle
\\&\quad +  2|(\nabla_{x}b)(X^{x_1,y_1}_t,y_1)\nabla_{y}X^{x_2,y_2}_t- (\nabla_{x}b)(X^{x_1,y_1}_t,y_2)\nabla_{y}X^{x_2,y_2}_t |\nabla_{y}X^{x_1,y_1}_t-\nabla_{y}X^{x_2,y_2}_t|
\\&\quad + 2|(\nabla_{x}b)(X^{x_1,y_1}_t,y_2)\nabla_{y}X^{x_2,y_2}_t- (\nabla_{x}b)(X^{x_2,y_2}_t,y_2)\nabla_{y}X^{x_2,y_2}_t| |\nabla_{y}X^{x_1,y_1}_t-\nabla_{y}X^{x_2,y_2}_t|
\\& \leq -2\gamma |\nabla_{y}X^{x_1,y_1}_t-\nabla_{y}X^{x_2,y_2}_t|^2
 + 2\|\nabla_{y}\nabla_{x}b\|_0 |\nabla_{y}X^{x_2,y_2}_t| |y_1-y_2| |\nabla_{y}X^{x_1,y_1}_t-\nabla_{y}X^{x_2,y_2}_t|
\\&\quad + 2\|\nabla_{x}^2b\|_0 |\nabla_{y}X^{x_2,y_2}_t| |X^{x_1,y_1}_t-X^{x_2,y_2}_t| |\nabla_{y}X^{x_1,y_1}_t-\nabla_{y}X^{x_2,y_2}_t|.
\end{aligned}
\end{equation*}
For the term $\Sigma_2$, we have
\begin{equation*}
\begin{aligned}
\Sigma_2 &\leq 2|(\nabla_{y}b)(X^{x_1,y_1}_t,y_1)-(\nabla_{y}b)(X^{x_1,y_1}_t,y_2)| |\nabla_{y}X^{x_1,y_1}_t-\nabla_{y}X^{x_2,y_2}_t|
\\&\quad + 2|(\nabla_{y}b)(X^{x_1,y_1}_t,y_2)-(\nabla_{y}b)(X^{x_2,y_2}_t,y_2)| |\nabla_{y}X^{x_1,y_1}_t-\nabla_{y}X^{x_2,y_2}_t|
\\& \leq 2 \|\nabla_{y}^2b\|_0 |y_1-y_2| |\nabla_{y}X^{x_1,y_1}_t-\nabla_{y}X^{x_2,y_2}_t|
\\&\quad + 2\|\nabla_{y}\nabla_{x}b\|_0 |X^{x_1,y_1}_t-X^{x_2,y_2}_t| |\nabla_{y}X^{x_1,y_1}_t-\nabla_{y}X^{x_2,y_2}_t|.
\end{aligned}
\end{equation*}
Obviously,  \eqref{xde} implies that
\begin{equation}
\label{deryb}
\sup_{t\ge0,x\in\mathbb{R}^{n},y\in\mathbb{R}^{m}}|\nabla_y X^{x,y}_t| \leq C(\|b\|_1,\gamma).
\end{equation}
Hence, by the assumption $b\in C^{2,2}_{b}$ and Young's inequality, we have
\begin{equation*}
\frac{d}{dt}|\nabla_{y}X^{x_1,y_1}_t-\nabla_{y}X^{x_2,y_2}_t|^2 \leq -\frac{\gamma}{2} |\nabla_{y}X^{x_1,y_1}_t-\nabla_{y}X^{x_2,y_2}_t|^2+ C(\|b\|_2,\gamma) \left(|y_1-y_2|^2+e^{-\frac{\gamma}{2}t}|x_1-x_2|^2\right).
\end{equation*}
By the comparison theorem, we obtain that
\begin{equation*}
|\nabla_{y}X^{x_1,y_1}_t-\nabla_{y}X^{x_2,y_2}_t|^2\leq C(\|b\|_2,\gamma) \left( t e^{-\frac{\gamma}{2}t}|x_1-x_2|^2+ |y_1-y_2|^2 \right).
\end{equation*}

\end{proof}

\begin{proof}[Proof of \eqref{xde2}]

Note that
\begin{equation*}
d\nabla_{x}X^{x,y}_t=(\nabla_{x} b)(X^{x,y}_t,y)\cdot\nabla_{x}X^{x,y}_tdt, ~~\nabla_{x}X^{x,y}_0=\bf{I}.
\end{equation*}
This implies
\begin{equation*}
\begin{aligned}
d(\nabla_{x}X^{x_1,y_1}_t-\nabla_{x}X^{x_2,y_2}_t)= &\left[(\nabla_{x}b)(X^{x_1,y_1}_t,y_1)\cdot \nabla_{x}X^{x_1,y_1}_t-(\nabla_{x}b)(X^{x_2,y_2}_t,y_2)\cdot\nabla_{x}X^{x_2,y_2}_t\right]dt.
\end{aligned}
\end{equation*}
Multiplying both sides by $2(\nabla_{x}X^{x_1,y_1}_t-\nabla_{x}X^{x_2,y_2}_t)$,  and using the inequality \eqref{disv}, we have
\begin{equation*}
\begin{aligned}
&\ \frac{d}{dt}|\nabla_{x}X^{x_1,y_1}_t-\nabla_{x}X^{x_2,y_2}_t|^2
\\
=&\ 2\left\langle (\nabla_{x}b)(X^{x_1,y_1}_t,y_1)\nabla_{x}X^{x_1,y_1}_t- (\nabla_{x}b)(X^{x_2,y_2}_t,y_2)\nabla_{x}X^{x_2,y_2}_t, \nabla_{x}X^{x_1,y_1}_t-\nabla_{x}X^{x_2,y_2}_t\right \rangle\\
\leq &\ 2\left\langle (\nabla_{x}b)(X^{x_1,y_1}_t,y_1)\nabla_{x}X^{x_1,y_1}_t- (\nabla_{x}b)(X^{x_1,y_1}_t,y_1)\nabla_{x}X^{x_2,y_2}_t, \nabla_{x}X^{x_1,y_1}_t-\nabla_{x}X^{x_2,y_2}_t\right\rangle
\\
&\ +  2|(\nabla_{x}b)(X^{x_1,y_1}_t,y_1)\nabla_{x}X^{x_2,y_2}_t- (\nabla_{x}b)(X^{x_1,y_1}_t,y_2)\nabla_{x}X^{x_2,y_2}_t| |\nabla_{x}X^{x_1,y_1}_t-\nabla_{x}X^{x_2,y_2}_t|
\\
&\ + 2|(\nabla_{x}b)(X^{x_1,y_1}_t,y_2)\nabla_{x}X^{x_2,y_2}_t- (\nabla_{x}b)(X^{x_2,y_2}_t,y_2)\nabla_{x}X^{x_2,y_2}_t| |\nabla_{x}X^{x_1,y_1}_t-\nabla_{x}X^{x_2,y_2}_t|
\\
\leq&\ -2\gamma|\nabla_{x}X^{x_1,y_1}_t-\nabla_{x}X^{x_2,y_2}_t|^2
+ 2\|\nabla_y\nabla_x b\|_0 |\nabla_{x}X^{x_2,y_2}_t| |y_1-y_2| |\nabla_{x}X^{x_1,y_1}_t-\nabla_{x}X^{x_2,y_2}_t|
\\
&\ + 2\|\nabla_x^2 b\|_0 |\nabla_{x}X^{x_2,y_2}_t| |X^{x_1,y_1}_t-X^{x_2,y_2}_t| |\nabla_{x}X^{x_1,y_1}_t-\nabla_{x}X^{x_2,y_2}_t|.
\end{aligned}
\end{equation*}
By \eqref{xde}, we have
\begin{equation*}
\sup_{t\ge0,x\in\mathbb{R}^{n},y\in\mathbb{R}^{m}}|\nabla_{x} X^{x,y}_t|^2 \leq e^{-\frac{\gamma t}{2}}.
\end{equation*}
Hence, by the assumption $b\in C^{2,2}_{b}$ and Young's inequality, we have
\begin{equation*}
\frac{d}{dt}|\nabla_{x}X^{x_1,y_1}_t-\nabla_{x}X^{x_2,y_2}_t|^2 \leq -\gamma |\nabla_{y}X^{x_1,y_1}_t-\nabla_{y}X^{x_2,y_2}_t|^2+ C(\|b\|_2,\gamma) e^{-\gamma t} \left(|y_1-y_2|^2+|x_1-x_2|^2\right),
\end{equation*}
and then the comparison theorem yields
\begin{equation*}
|\nabla_{x}X^{x_1,y_1}_t-\nabla_{x}X^{x_2,y_2}_t|^2\leq C(\|b\|_2,\gamma) te^{-\frac{\gamma t}{2}} \left(|y_1-y_2|^2+|x_1-x_2|^2\right).
\end{equation*}

\end{proof}

\subsection{Proof of Lemma \ref{ps}.}

\begin{proof}[Proof of \eqref{eqn-2}]
Set
\begin{equation*}
\bar{G}(y)=\int_{\mathbb{R}^n}G(x,y)\mu^{y}(dx).
\end{equation*}
Then Hypothesis ($\bf{A_{G2}}$) yields $\bar G\equiv0$. By Proposition \ref{ergodic11} and Hypothesis ($\bf{A_{G1}}$),  we have
\begin{equation*}
\begin{aligned}
|\widetilde{{G}}(x,y)|&\leq \int^{\infty}_{0}\left|\mathbb{E}[G(X^{x,y}_t,y)]-\bar{G}(y)\right|dt,
 \leq C(1+|x|^{\frac{1}{2}})\int^{\infty}_{0}e^{-\frac{\gamma t}{8}}dt
 \leq C(1+|x|^{\frac{1}{2}}).
\end{aligned}
\end{equation*}
\end{proof}

\begin{proof}[Proof of \eqref{eqn-3}]
Note that
\begin{equation*}
\nabla_{x}\widetilde{{G}}(x,y)=\int^{\infty}_{0}\mathbb{E}\left[\nabla_{x}G(X^{x,y}_t,y)\cdot \nabla_{x}X^{x,y}_t \right]dt,
\end{equation*}
where $\nabla_{x}X^{x,y}_t$ satisfies
\begin{equation*}
\left\{
\begin{aligned}
&d\nabla_{x} X^{x,y}_t=\nabla_{x}b(X^{x,y}_t,y)\cdot \nabla_{x} X^{x,y}_tdt, \\
&\nabla_{x} X^{x,y}_t|_{t=0}=I.
\end{aligned}
\right.
\end{equation*}
By \eqref{xde}, we have
\begin{equation*}
\sup_{x,y}|\nabla_{x} X^{x,y}_t|\leq Ce^{-\frac{\gamma t}{4}}.
\end{equation*}
Thus by Hypothesis ($\bf{A_{G1}}$), we have
\begin{equation*}
\sup_{x,y}|\nabla_{x} \widetilde{{G}}(x,y)|\leq C.
\end{equation*}
\end{proof}

\begin{proof}[Proof of \eqref{eqn-4}]
Set
\begin{equation*}
\widetilde{G}_{t_0}(x,y,t):=\mathbb{E}G(X^{x,y}_t, y)-\mathbb{E}G(X^{x,y}_{t+{t_0}},y) =:\widehat{G}(x,y,t)-\widehat{G}(x,y,t+{t_0}).
\end{equation*}
Then  Proposition \ref{ergodic11} implies that
\begin{equation}\label{est-5}
\lim_{t_0\rightarrow \infty} \widetilde{G}_{{t_0}}(x,y,t)=\mathbb{E}G(X^{x,y}_t, y)-\bar{G}(y) = \mathbb{E}G(X^{x,y}_t, y).
\end{equation}
On the one hand, by the Markov property, we have
\begin{equation*}
\widetilde{G}_{t_0}(x,y,t)=\widehat{G}(x,y,t)-\mathbb{E}\widehat{G}(X^{x,y}_{t_0},y,t).
\end{equation*}
Thus we have
\begin{equation*}
\begin{aligned}
\nabla_{y}\widetilde{G}_{t_0}(x,y,t)&=\nabla_{y}\widehat{G}(x,y,t)-\mathbb{E}\left[\nabla_{y}\widehat{G}(X^{x,y}_{t_0},y,t)\right]
-\mathbb{E}\left[\nabla_{x}\widehat{G}(X^{x,y}_{t_0},y,t)\cdot \nabla_y X^{x,y}_{t_0}\right].
\end{aligned}
\end{equation*}
Moreover, we also have
\begin{equation*}
\nabla_{x}\widehat{G}(x,y,t)=\mathbb{E}\left[\nabla_{x}{G}(X^{x,y}_t,y)\cdot \nabla_{x} X^{x,y}_t\right].
\end{equation*}
By Hypothesis ($\bf{A_{G1}}$) and \eqref{xde},  we have
\begin{equation}\label{est-4}
\sup_{x,y}|\nabla_{x}\widehat{G}(x,y,t)|\leq C e^{-\frac{\gamma t}{4}}.
\end{equation}
On the other hand, we have
\begin{equation}\label{est-2}
\begin{aligned}
|\nabla_{y}\widehat{G}(x_1,y,t)-\nabla_{y}\widehat{G}(x_2,y,t)|&=|\nabla_{y} \left(\mathbb{E}G(X^{x_1,y}_t,y)\right)-\nabla_{y} \left(\mathbb{E}G(X^{x_2,y}_t,y)\right)|\\
&=\mathbb{E}|\nabla_{x} G(X^{x_1,y}_t, y)\cdot \nabla_{y} X^{x_1,y}_t-\nabla_{x} G(X^{x_2,y}_t, y)\cdot \nabla_{y} X^{x_2,y}_t| \\
&\quad +\mathbb{E}|\nabla_{y} G(X^{x_1,y}_t, y)-\nabla_{y} G(X^{x_2,y}_t, y)|\\
& \leq  \mathbb{E}|\nabla_{x} G(X^{x_1,y}_t, y)\cdot \nabla_{y} X^{x_1,y}_t-\nabla_{x} G(X^{x_2,y}_t, y)\cdot \nabla_{y} X^{x_1,y}_t|\\
&\quad +\mathbb{E}|\nabla_{x} G(X^{x_2,y}_t, y)\cdot \nabla_{y} X^{x_1,y}_t-\nabla_{x} G(X^{x_2,y}_t, y)\cdot \nabla_{y} X^{x_2,y}_t|\\
&\quad +\mathbb{E}|\nabla_{y} G(X^{x_1,y}_t,y)-\nabla_{y} G(X^{x_2,y}_t,y)|\\
&:=S_1+S_2+S_3.
\end{aligned}
\end{equation}
For the term $S_1$, by the boundedness of $\nabla_{x}G $ and $\nabla_{x}\nabla_{y}G$ in Hypothesis ($\bf{A_{G1}}$), \eqref{deryb} and \eqref{xde}, we have
\begin{equation*}
S_1 \leq C \mathbb{E}\left[|X^{x_1,y}_{t}-X^{x_2,y}_{t}|\right]^{1/2}\leq C e^{-\frac{\gamma t}{8}}|x_1-x_2|^{\frac{1}{2}}.
\end{equation*}
For the term $S_2$, by the boundedness of $\nabla_{x}G$ and \eqref{xde1}, we have
\begin{equation*}
S_2 \leq C \mathbb{E}\left[|\nabla_{y} X^{x_1,y}_t-\nabla_{y} X^{x_2,y}_t|\right]^{1/2}\leq C t^{\frac{1}{4}} e^{-\frac{\gamma t}{8}}|x_1-x_2|^{\frac{1}{2}}.
\end{equation*}
For the term $S_3$, by the boundedness of $\nabla_{y}G $ and $\nabla_{x}\nabla_{y}G$ and \eqref{xde}, we have
\begin{equation*}
S_3 \leq C \mathbb{E}\left[|X^{x_1,y}_t-X^{x_2,y}_t|\right]^{1/2}\leq C e^{-\frac{\gamma t}{8}}|x_1-x_2|^{\frac{1}{2}}.
\end{equation*}
Combining these together, we achieve from \eqref{est-2} that
\begin{equation}\label{est-3}
|\nabla_{y} \widehat{G}(x_1,y,t)-\nabla_{y}\widehat{G}(x_2,y,t)|\leq C(1+t^{\frac{1}{4}})e^{-\frac{\gamma t}{8}}|x_1-x_2|^{\frac{1}{2}}.
\end{equation}
Therefore, by \eqref{est-4}, \eqref{est-3}, \eqref{xde} and \eqref{moment0},  we have
\begin{equation*}
\begin{aligned}
\left|\nabla_{y} \widetilde{G}_{t_0}(x,y,t)\right|&=\left|\mathbb{E}\left[\nabla_{y} \widetilde{G}(x,y,t)-\nabla_{y} \widetilde{G}(X^{x,y}_{t_0},y,t)\right]-\mathbb{E}\left[\nabla_{x} \widehat{G}(X^{x,y}_{t_0},x,y)\cdot \nabla_{y} X^{x,y}_{t_0}\right]\right|\\
& \leq C(1+t^{\frac{1}{4}})e^{-\frac{\gamma t}{8}} \mathbb{E}\left[|X^{x,y}_{t_0}-x|\right]^{\frac{1}{2}}+Ce^{-\frac{\gamma t}{8}}
 \leq C(1+t^{\frac{1}{4}}) e^{-\frac{\gamma t}{8}}(1+|x|^{\frac{1}{2}}).
\end{aligned}
\end{equation*}
This together with \eqref{est-5} implies that
\begin{equation*}
\begin{aligned}
\left|\nabla_{y} \widetilde{{G}}(x,y)\right|&=\left|\int^{\infty}_{0}\left(\lim_{t_0\rightarrow \infty}\nabla_{y} \widetilde{G}_{t_0}(x,y,t)\right)dt\right|
 \leq \int^{\infty}_{0} C(1+t^{\frac{1}{4}})e^{-\frac{\gamma t}{8}}(1+|x|^{\frac{1}{2}})dt
 \leq C(1+|x|^{\frac{1}{2}}).
\end{aligned}
\end{equation*}
\end{proof}

\begin{proof}[Proof of \eqref{eqn-5}]
Note that by \eqref{est-5},
\begin{equation*}
|\nabla_{y}^2 \widetilde{{G}}(x,y)|=\left|\int^{\infty}_{0}\left(\lim_{t_0\rightarrow \infty}\nabla_{y}^2\widetilde{G}_{t_0}(x,y,t)\right)dt\right|,
\end{equation*}
and
\begin{equation*}
\begin{aligned}
\nabla_{y}^2\widetilde{G}_{t_0}(x,y,t)&=\mathbb{E}\left[\nabla_{y}^2\widehat{G}(x,y,t)-\nabla_{y}^2\widehat{G}(X^{x,y}_{t_0},y,t)\right]
-\mathbb{E}\left[\nabla_{x}\nabla_{y}\widehat{G}(X^{x,y}_{t_0},y,t)\cdot\nabla_{y}X^{x,y}_{t_0}\right]\\
&\quad -\mathbb{E}\left[\nabla^2_{x}\widehat{G}(X^{x,y}_{t_0},y,t)\cdot\left(\nabla_{y}X^{x,y}_{t_0}\right)^2\right]
-\mathbb{E}\left[\nabla_{y}\nabla_{x}\widehat{G}(X^{x,y}_{t_0},y,t)\cdot\nabla_{y}X^{x,y}_{t_0}\right]
\\
&\quad -\mathbb{E}\left[\nabla_{x}\widehat{G}(X^{x,y}_{t_0},y,t)\cdot\nabla^2_{y}X^{x,y}_{t_0}\right]\\
&=:T_1-T_2-T_3-T_4-T_5.
\end{aligned}
\end{equation*}
where $\nabla_{y}X^{x,y}_t$ satisfies
\begin{equation*}
\left\{
\begin{aligned}
&d\nabla_{y} X^{x,y}_t=\nabla_{x}b(X^{x,y}_t,y)\cdot \nabla_{y} X^{x,y}_tdt+\nabla_{y} b(X^{x,y}_t,y)dt, \\
&\nabla_{y} X^{x,y}_t|_{t=0}={\mathbf{0}},
\end{aligned}
\right.
\end{equation*}
and $\nabla^2_{y}X^{x,y}_t$ satisfies
\begin{equation*}
\left\{
\begin{aligned}
d\nabla^2_{y}X^{x,y}_{t}&=\nabla^2_{x}b(X^{x,y}_{t},y)\cdot (\nabla_{y}X^{x,y}_t)^2dt +\nabla_{y}(\nabla_{x}b)(X^{x,y}_{t},y)\cdot \nabla_{y}X^{x,y}_tdt\\
&\quad +(\nabla_{x}b)(X^{x,y}_{t},y)\cdot \nabla^2_{y}X^{x,y}_t dt
+\nabla_{x}(\nabla_{y}b)(X^{x,y}_{t},y)\cdot \nabla_{y}X^{x,y}_tdt+\nabla_{y}^2b(X^{x,y}_{t},y)dt,\\
\nabla^2_{y}X^{x,y}_{t}|_{t=0}&=\mathbf{0}.
\end{aligned}
\right.
\end{equation*}

To estimate the term $T_1$, we recall that $\widehat{G}(x,y,t) =\mathbb{E}G(X^{x,y}_t, y)$. We derive
%\begin{equation}
%\left\{
%\begin{aligned}
%\nabla_{y}G(X^{x,y}_t,y)&=\nabla_{x}G(X^{x,y}_t,y)\cdot\nabla_{y} X^{x,y}_t+\nabla_{y}G(X^{x,y}_t,y),\\
%\nabla_{y}^2G(X^{x,y}_t,y)&=\nabla_{x}^2G(X^{x,y}_t,y)\cdot\left(\nabla_{y} X^{x,y}_t\right)^2
%+2\nabla_{x}\nabla_{y}G(X^{x,y}_t,y)\cdot\nabla_{y} X^{x,y}_t\\
%&+\nabla_{x}G(X^{x,y}_t,y)\cdot\partial^2_y X^{x,y}_t+\partial^2_2G(X^{x,y}_t,y).
%\end{aligned}
%\right.
%\end{equation}
%Then we have
\begin{equation*}
\begin{aligned}
\nabla_{y}^2(G(X^{x_1,y}_t,y))-\nabla_{y}^2(G(X^{x_2,y}_t,y))&= \left[ \left(\nabla_{y} X^{x_1,y}_t \right)^T \cdot \nabla_{x}^2G(X^{x_1,y}_t,y)\cdot \nabla_{y} X^{x_1,y}_t - \left(\nabla_{y} X^{x_2,y}_t\right)^T \cdot \nabla_{x}^2G(X^{x_2,y}_t,y)\cdot \nabla_{y} X^{x_2,y}_t \right]\\
&\quad +2\left[\nabla_{x}\nabla_{y}G(X^{x_1,y}_t,y)\cdot\nabla_{y} X^{x_1,y}_t-\nabla_{x}\nabla_{y}G(X^{x_2,y}_t,y)\cdot\nabla_{y} X^{x_2,y}_t\right]\\
&\quad +\left[\nabla_{y}^2G(X^{x_1,y}_t,y)-\nabla_{y}^2G(X^{x_2,y}_t,y)\right]\\
&=:T_{11}+T_{12}+T_{13}.
\end{aligned}
\end{equation*}
For the term $T_{11}$, we use \eqref{xde} and \eqref{xde1} to get
\begin{equation}
\label{t11}
\begin{aligned}
\mathbb{E}\left(|T_{11}|\right)&\leq \mathbb{E}\left[\left|\left(\nabla_{y} X^{x_1,y}_t\right)^T \cdot \nabla_{x}^2G(X^{x_1,y}_t,y)\cdot \nabla_{y} X^{x_1,y}_t - \left(\nabla_{y} X^{x_1,y}_t\right)^T \cdot \nabla_{x}^2G(X^{x_2,y}_t,y) \cdot \nabla_{y} X^{x_1,y}_t \right|\right]\\
&\quad +\mathbb{E}\left[\left| \left(\nabla_{y} X^{x_1,y}_t\right)^T \cdot \nabla_{x}^2G(X^{x_2,y}_t,y) \cdot \nabla_{y} X^{x_1,y}_t - \left(\nabla_{y} X^{x_1,y}_t\right)^T \cdot \nabla_{x}^2G(X^{x_2,y}_t,y) \cdot \nabla_{y} X^{x_2,y}_t \right|\right] \\
&\quad +\mathbb{E}\left[\left| \left(\nabla_{y} X^{x_1,y}_t\right)^T \cdot \nabla_{x}^2G(X^{x_1,y}_t,y) \cdot \nabla_{y} X^{x_1,y}_t - \left(\nabla_{y} X^{x_2,y}_t\right)^T \cdot \nabla_{x}^2G(X^{x_2,y}_t,y) \cdot \nabla_{y} X^{x_2,y}_t \right|\right] \\
& \leq \|\nabla^3_{x}G\|_0 \mathbb{E}\left[|X^{x_1,y}_{t}-X^{x_2,y}_{t}| |\nabla_{y} X^{x_1,y}_t|^2\right] + \|\nabla_{x}^2G\|_0 \mathbb{E}\left[ (\left| \nabla_{y} X^{x_1,y}_t\right| + \left| \nabla_{y} X^{x_2,y}_t\right|) \left|\nabla_{y} X^{x_1,y}_t - \nabla_{y} X^{x_2,y}_t \right|\right] \\
& \leq C(1+t^{\frac{1}{2}})e^{-\frac{\gamma t}{4}}|x_1-x_2|.
\end{aligned}
\end{equation}
Similarly, for the term $T_{12}$ and $T_{13}$,  we  have
\begin{equation}\label{t12}
\begin{aligned}
\mathbb{E}\left(|T_{12}|\right)&\leq 2 \|\nabla_{x}^2\nabla_{y} G\|_0 \mathbb{E}\left[ |X^{x_1,y}_t-X^{x_2,y}_t| |\nabla_{y} X^{x_1,y}_t|\right]
 +2 \|\nabla_{x}\nabla_{y}G\|_0 \mathbb{E}\left[|\nabla_{y} X^{x_1,y}_t-\nabla_{y} X^{x_2,y}_t|\right]\\
& \leq C(1+t^{\frac{1}{2}})e^{-\frac{\gamma t}{4}}|x_1-x_2|,
\end{aligned}
\end{equation}
and
\begin{equation}
\label{t13}
\mathbb{E}\left(|T_{13}|\right) \leq \|\nabla_x \nabla_{y}^2G\|_0 |X^{x_1,y}_t-X^{x_2,y}_t| \le Ce^{-\frac{\gamma t}{4}}|x_1-x_2|.
\end{equation}
Combining \eqref{t11}--\eqref{t13}, we obtain
\begin{equation}
\label{T1}
|T_1| \leq C(1+t^{\frac{1}{2}})e^{-\frac{\gamma t}{4}}\mathbb{E}|X^{x,y}_{t_0}-x|\leq C (1+t^{\frac{1}{2}})e^{-\frac{\gamma t}{4}}(1+|x|).
\end{equation}

For the term $T_2$, again by the definition of $\widehat{G}(x,y,t)$, we have
\begin{equation*}
\nabla_{x}\nabla_{y}\widehat{G}(x,y,t) = \mathbb{E}\left[ (\nabla_{x}X^{x,y}_t)^T \cdot \nabla_{x}^2G(X^{x,y}_{t},y)\cdot \nabla_{y}X^{x,y}_t
+\nabla_{x}G(X^{x,y}_{t},y)\cdot\nabla_{x}\nabla_{y}X^{x,y}_t
+\nabla_{x}\nabla_{y}G(X^{x,y}_{t},y)\cdot\nabla_{x}X^{x,y}_{t} \right].
\end{equation*}
By \eqref{xde} and \eqref{xde1}, we get
\begin{equation*}
\begin{aligned}
|\nabla_{x}\nabla_{y}\widehat{G}(x,y,t)|& \leq C(1+t^{\frac{1}{2}})e^{-\frac{\gamma t}{4}}.
\end{aligned}
\end{equation*}
Therefore we have
\begin{equation*}
|T_2|\leq C( 1+ t^{\frac{1}{2}})e^{-\frac{\gamma t}{4}}.
\end{equation*}

For the term $T_3$, we have
%\begin{equation}
%\left\{
%\begin{aligned}
%\nabla_{x}G(X^{x,y}_{t},y)&=\nabla_{x}G(X^{x,y}_{t},y)\cdot \nabla_{x}X^{x,y}_{t},\\
%\nabla^2_{x}G(X^{x,y}_{t},y)&=\nabla_{x}^2G(X^{x,y}_{t},y)\cdot \left(\nabla_{x}X^{x,y}_{t}\right)^2+\nabla_{x}G(X^{x,y}_{t},y)\cdot\partial^2_{x}X^{x,y}_{t}.
%\end{aligned}
%\right.
%\end{equation}
%Moreover, we have
\begin{equation*}
\begin{aligned}
\nabla_{x}^2\widehat{G}(x,y,t)=\mathbb{E}\left[ \left(\nabla_{x} X^{x,y}_t\right)^T \cdot \nabla_{x}^2G(X^{x,y}_t,y)\cdot \nabla_{x} X^{x,y}_t
+\nabla_{x} G(X^{x,y}_t,y)\cdot \nabla^2_{x}X^{x,y}_t\right].
\end{aligned}
\end{equation*}
By \eqref{xde} and \eqref{xde2}, we obtain
\begin{equation*}
|T_3| \leq C\left(e^{-\frac{\gamma t}{2}} +t^{\frac{1}{2}} e^{-\frac{\gamma t}{4}} \right).
\end{equation*}

For the term $T_4$, we have
\begin{equation*}
\begin{aligned}
\nabla_{y}\nabla_{x}\widehat{G}(x,y,t)&= \mathbb{E}\left[(\nabla_{y} X^{x,y}_t)^T \cdot \nabla_{x}^2G(X^{x,y}_t,y)\cdot \nabla_{x}  X^{x,y}_t
+\nabla_{y} \nabla_{x}G(X^{x,y}_t,y)\cdot \nabla_{x} X^{x,y}_t + \nabla_{x} G(X^{x,y}_t,y)\cdot \nabla_{y}\nabla_{x}X^{x,y}_t\right].
\end{aligned}
\end{equation*}
By \eqref{xde} and \eqref{xde2}, we have
\begin{equation*}
|T_4|\leq C( 1+ t^{\frac{1}{2}})e^{-\frac{\gamma t}{4}}.
\end{equation*}

For the term $T_5$, we  have
\begin{equation*}
\begin{aligned}
\nabla_{x} \widehat{G}(x,y,t)=\mathbb{E}\left[ \nabla_{x}G(X^{x,y}_t,y)\cdot \nabla_{x} X^{x,y}_t \right].
\end{aligned}
\end{equation*}
By \eqref{xde} and \eqref{xde1}, we get
\begin{equation}
\label{T5}
|T_5| \leq Ce^{-\frac{\gamma t}{4}}.
\end{equation}
Combining  \eqref{T1} and \eqref{T5}, we obtain
\begin{equation*}
|\nabla_{y}^2 \widetilde{{G}}(x,y)|\leq C \int_0^\infty \left[ (1+t^{\frac{1}{2}})e^{-\frac{\gamma t}{4}}(1+|x|) + e^{-\frac{\gamma t}{2}} \right] dt \le C(1+|x|).
\end{equation*}
\end{proof}

\section{Acknowledgments.}
\medskip
\textbf{Acknowledgements}.  We would like to thank Dr.~Wei Wei (Huazhong University of Science and Technology, China) and Prof.~Xiaobin Sun (Jiangsu Normal University, China) for their helpful discussions. The research of Y. Zhang was supported by the NSFC grant 11901202. The work of Q. Huang was supported by FCT, Portugal, project PTDC/MAT-STA/28812/2017. The research
of X. Wang was supported by the NSFC grant 11901159 and is supported by the Natural Science Foundation of Henan Province of China
(Grant No. 232300420110). The research of Z. Wang was supported by the NSFC grant. 11531006 and 11771449.

%\section*{Reference}


\begin{thebibliography}{0}





\bibitem{GA}
G. A. Pavliotis, A. M. Stuart, Multiscale methods: averaging and homogenization,
Springer, New York,
2008.


\bibitem{GA1}
A. Papavasilion,
Coarse-grained modeling of multiscale diffusion: The p-variation estimates,
Spring, Heidelberg,
2011.





\bibitem{JA}
A. J. Majda, I. Timofeyev, E. V. Eijnden,
A mathematical framework for stochastic climate models,
Comm. Pure Apppl. Math.,
54 (2001) 891-974.



\bibitem{MK}
M. Katsoulakis, A. Majda, A. Sopasakis,
Multiscale couplings in prototype hybrid deterministic/stochastic systems: Part I, deterministic closures,
Commun. Math. Sci.,
2 (2004) 255-294.




\bibitem{Huang}
Q. Huang, J. Duan, R. Song,
Homogenization of nonlocal partial differential equations related to stochastic differential equations with L\'evy noise, Bernoulli, 28 (2022) 1648-1674.


\bibitem{Huang23}
Q. Huang, J. Duan, R. Song,
Homogenization of non-symmetric jump processes, Adv. Appl. Probab. (2023): 1-33.




\bibitem{KR}
R. Z. Khasminski,
A limit theorem for solutions of differential equations with random right-hand side,
Theory Probab. Appl.,
11 (1966) 390-406.




\bibitem{SD}
G. C. Papanicolaou, D. W. Stroock, S. R. S. Varadhan,
Martingale approach to some limit theorem,
In Conference on Statistical Mechanics,
Dynamical Syatems and Turbulence, Duke Univ. Press, 1977.


%\bibitem{ET}
%S. N. Ethier, T. G. Kurtz,
%Markov Processes: Characterization and Convergence,
%Wiley, New York,
%1986.





\bibitem{EP}
E. Pardoux, A. Y. Veretennilov,
On the Poisson equation and diffusion approximation. I,
Ann. Probab.,
29 (2001) 1061-1085.











\bibitem{WE}
W. E, B. Engquist,
Analysis of multiscale methods for stochastic differential equations,
Commun. Pur. Appl. Math.,
58 (2005) 1544-1585.



\bibitem{SM}
S. Cerrai, M. Freidlin,
Averaging principle for a class of stochastic reaction-diffusion equations,
Probab. Theory Relat. Fields,
144 (2009) 137-177.




\bibitem{MHL}
 M. Hairer, X.-M. Li,
 Averaging dynamics driven by fractional Brownian motion,
 Ann. Probab.,
48 (2020) 1826-1860.



\bibitem{Rock}
 M. R\"ockner, L. J. Xie,
 Diffusion approximation for fully coupled stochastic differential equations,
Ann. Probab.,
49 (2021) 1205-1236.



%\bibitem{Rock11}
% M. R\"ockner, L. J. Xie,
% Averaging principle and normal deviations for multiscale stochastic systems,
%Comm. Math. Phys.,
% 383 (2021) 1889-1937.









\bibitem{bg}
J. Bao, G. Yin, C. Yuan,
Two-time-scale stochastic partial differential equations driven by $\alpha$-stable noises: Averaging principles,
Bernoulli,
23 (2018) 645-669.




\bibitem{YZ}
Y. Zhang, Z. Cheng, X. Zhang, X. Chen, J. Duan, X. Li,
Data assimilation and parameter estimation
for a multiscale stochastic system with $\alpha$-stable L\'evy noise,
J. Stat. Mech. Theory Exp.,
11 (2017) 113401.





%\bibitem{XS}
%X. Sun, J. Zhai,
%Averaging principle for stochastic real Ginzburg-Landau equation driven by $\alpha$-stable process,
%Commun. Pure Appl. Anal.,
%19 (2020) 1291-1319.






\bibitem{XS1}
X. Sun, L. Xie, Y. Xie,
Strong and weak convergent rates for slow-fast stochastic differential equations driven by $\alpha$-stable process, Bernoulli,  28 (2022) 343-369.



\bibitem[Sato(1999)]{Sat99}
K.-I. Sato,
L{\'e}vy processes and infinitely divisible distributions,
Cambridge University Press, Cambridge, 1999.



\bibitem{BSW13}
B. B\"ottcher, R. Schilling, J. Wang. L\'evy matters III, Lecture Notes in Mathematics 2099, Springer International Publishing, Switzerland, 2013.



\bibitem{ar}
S. Albeverio, B. R\"udiger, J. Wu,
Invariant measures and symmetry property of L\'evy type operators,
Potential Anal.,
13 (2000) 147-168.


\bibitem{Da}
D. Applebaum,
L\'evy Processes and Stochastic Calculus,
Cambridge University Press, Cambridge, 2004.




\bibitem{Maj}
M. B. Majka,
Coupling and exponential ergodicity for stochastic differential equations driven by L\'evy processes,
Stoch. Process. Their Appl.,
127 (2017) 4083-4125.




\bibitem{EN000}
K. -J. Engel, R. Nagel,
One-parameter semigroups for linear evolution equations,
 Springer-Verlag, New York, 2000.






\bibitem[Jacod and Shiryaev(2013)]{JS13}
J.~Jacod, A.~Shiryaev,
Limit theorems for stochastic processes,
Springer,  Berlin, 2013.




\bibitem[Peszat(2007)]{PZ07}
S.~Peszat J.~Zabczyk,
Stochastic partial differential equations with {L\'e}vy noise: An evolution equation approach,
Cambridge University Press, 2007.









%\bibitem{ar}
%S. Albeverio, B. R\"udiger and J. Wu,
%Invariant measures and symmetry property of L\'evy type operators,
%Potential Anal., 13(2000), pp. 147-168.
%
%\bibitem{Maj17}
%M. B. Majka,
%Coupling and exponential ergodicity for stochastic
%differential equations driven by L\'evy processes,
%Stochastic Processes and their Applications, 127(2017), pp. 4083¨C4125.
%











%\bibitem{pl}
%P. Lachout,
%A skorohod space of discontinuous functions on a general set,
%Acta Univ. Carolinae. Math. Phys.,
%33 (1992) 91-97.






\bibitem[Billingsley(1968)]{Bil68}
P.~Billingsley,
Convergence of Probability Measures,
John Wiley \& Sons, Inc., 1968.




\bibitem{rc}
R. Cont, D. A. Fourni\'e,
Change of variable formulas for non-anticipative functionals on path space,
J. Funct. Anal.,
259 (2010) 91-97.






%\bibitem {Bas09}
%R.F. Bass.
%\newblock Regularity results for stable-like operators.
%\newblock \emph{J. Funct. Anal.}, 8\penalty0 (257):\penalty0 2693--2722, 2009.
%
%
%
%
%
%\bibitem[Priola(2012)]{Pri12}
%E.~Priola.
%\newblock Pathwise uniqueness for singular {SDEs} driven by stable processes.
%\newblock \emph{Osaka J. Math.}, 49\penalty0 (2):\penalty0 421--447, 2012.
%
%
%
%
%
%
%
%
%\bibitem[Gilbarg and Trudinger(2001)]{GT01}
%D.~Gilbarg and N.S. Trudinger.
%\newblock \emph{Elliptic partial differential equations of second order}, vol.
%  224.
%\newblock Springer Science \& Business Media, 2001.















\bibitem[czhang(2016)]{czhang}
Z.-Q. Chen and X. Zhang,
Heat kernels and analyticity of non-symmetric jump diffusion semigroups,
Probab. Theory Relat. Fields, 65 (2016) 267-312.









\end{thebibliography}
\end{document}